\documentclass[letterpaper, 10 pt, conference]{ieeeconf}  

\IEEEoverridecommandlockouts                              

\usepackage{siunitx} 
\usepackage{amsmath,amssymb,amsfonts,mathtools,tensor,bm} 
\usepackage{cite}
\usepackage{graphicx}
\usepackage{xcolor}
\usepackage{caption}
\usepackage{fancyhdr}
\usepackage{booktabs}
\usepackage{tabu}
\usepackage{multirow}
\usepackage{tabularx}

\usepackage{pgfplots}
\usepackage{varwidth}

\usepackage[ruled,vlined,onelanguage]{algorithm2e}
\SetKwProg{Task}{}{:}{}
\SetKwProg{Fn}{Function}{:}{}

\usepackage[noconfig,amsmath,thmmarks]{ntheorem}
\theoremnumbering{arabic}
\theoremstyle{plain}
\RequirePackage{latexsym}
\theorembodyfont{\itshape}
\theoremheaderfont{\normalfont\bfseries}
\theoremseparator{}

\newtheorem{lemma}{Lemma}

\newtheorem{remark}{Remark}
\newtheorem{problem}{Problem}

\newtheorem{assumption}{Assumption}

\newcommand{\cov}[2]{\mathrm{cov}\left( #1, #2\right)}
\newcommand{\covs}[1]{\mathrm{cov}\left( #1 \right)}
\newcommand{\E}[1]{\mathrm{E}\left\{#1\right\}}
\newcommand{\mat}[1]{\left[\begin{matrix} #1 \end{matrix}\right]}
\newcommand{\norm}[1]{\left\lVert#1\right\rVert}
\DeclareMathOperator{\vecv}{vec}
\DeclareMathOperator{\diag}{diag}
\DeclareMathOperator{\ISOC}{ISOC}

\usepackage{ifthen}
\newboolean{showcomments}
\setboolean{showcomments}{false} 
\ifthenelse{\boolean{showcomments}}{ 
	\newcommand\td[1]{\textcolor{red}{\textbf{TODO:} #1}} 
	\newcommand\tdd[1]{} 
	\newcommand\comment[1]{\textcolor{blue}{\textbf{Comment:} #1}} 
	\newcommand\commentd[1]{} 
	\newcommand\frage[1]{\textcolor{orange}{\textbf{Rückfrage:} #1}} 
	\newcommand\fraged[1]{}
}{ 
	\newcommand\td[1]{}
	\newcommand\tdd[1]{} 
	\newcommand\comment[1]{} 
	\newcommand\commentd[1]{} 
	\newcommand\frage[1]{}
	\newcommand\fraged[1]{}
}

\title{\LARGE \bf
Inverse Stochastic Optimal Control for Linear-Quadratic Gaussian and Linear-Quadratic Sensorimotor Control Models
}

\author{Philipp Karg, Simon Stoll, Simon Rothfuß and Sören Hohmann*
\thanks{*All authors are with the Institute of Control Systems (IRS) at the Karlsruhe Institute of Technology (KIT), 76131 Karlsruhe, Germany. Corresponding author is Philipp Karg, {\tt\small philipp.karg@kit.edu}.}
}

\begin{document}

\maketitle
\thispagestyle{fancy}
\pagestyle{fancy}

\fancyhf{}
\fancyhead[CO,CE]{\copyright~2022 IEEE. This paper has been accepted for publication and presentation at the \\61st IEEE Conference on Decision and Control.}


\begin{abstract}

In this paper, we define and solve the Inverse Stochastic Optimal Control (ISOC) problem of the linear-quadratic Gaussian (LQG) and the linear-quadratic sensorimotor (LQS) control model. These Stochastic Optimal Control (SOC) models are state-of-the-art approaches describing human movements. The LQG ISOC problem consists of finding the unknown weighting matrices of the quadratic cost function and the covariance matrices of the additive Gaussian noise processes based on ground truth trajectories observed from the human in practice. The LQS ISOC problem aims at additionally finding the covariance matrices of the signal-dependent noise processes characteristic for the LQS model. We propose a solution to both ISOC problems which iteratively estimates cost function and covariance matrices via two bi-level optimizations. Simulation examples show the effectiveness of our developed algorithm. It finds parameters that yield trajectories matching mean and variance of the ground truth data.
\comment{bewusst ground truth statt measured genutzt - ground truth allgemeiner, womit in Simulation simulierte Trajektorien und in Praxis gemessene eingeschlossen sind - formal wird das dann in Remark~\ref{remark:DefinitionReference} definiert und auch bzgl. der Anwendung auf Messdaten glatt gezogen - ground truth im Vgl. zu reference bevorzugen wg. Möglichkeit, dass ground truth data aus reference tracking resultieren - ground truth = vom Menschen gemessene Daten, sprich keine virtuellen Größen, wie virtuelle Eingangsdaten oder Ausgangsdaten, vorhanden sowie i.Allg. nicht alle Systemzustände messbar}

\end{abstract}


\section{INTRODUCTION} \label{sec:intro}
Stochastic Optimal Control (SOC) models are state-of-the-art approaches describing human movements to a single goal \cite{Gallivan.2018}. They characterize the average behavior via the mean as well as the variability patterns via higher stochastic moments of the system state. The current main representative is the linear-quadratic (LQ) sensorimotor (LQS) control model \cite{Todorov.2002, Todorov.2004, Todorov.2005} which builds upon the well-known LQ Gaussian (LQG) control model but takes signal-dependent noise processes in human execution and perception into account. Compared to the LQG case, in the LQS model a control-dependent noise process is added to the system state equation and a state-dependent one to the output equation. In order to verify the model hypothesis and investigate the optimality principles underlying human movements an identification of the unknown parameters of the SOC models, namely relative weights of the cost function and covariance matrices of the noise processes, on the basis of ground truth data, i.e. human measurements, is needed. The identified SOC models can then build the starting point to design a supporting automation in human-machine collaboration settings. For example, by taking advantage of the predicted human variability the accuracy of executing a via-point-movement task can be improved. In case of the LQS model, the identification of the covariance matrices is not only necessary to describe human variability patterns, like in the LQG case, but also to model human average behavior (see our results of Section~\ref{sec:isoc_problems})\footnotemark. However, research regarding the inverse problem of LQ SOC models, i.e. identifying both unknown parameter types, weighting matrices of the cost function and covariance matrices of the noise processes, is scarce.

\footnotetext{If only additive Gaussian noise is present, its corresponding covariance matrices solely influence the average behavior in case of nonlinear system dynamics \cite{Berret.2021}.}

Among the approaches considering the identification problem of SOC models, \cite{Priess.2014} address only the LQG model and thus, the influence of the signal-dependent noise processes cannot be investigated. The same holds for \cite{Chen.2015} and in addition, the proposed method determines solely cost function parameters. In \cite{Li.2011}, the control-dependent noise process is taken into account. However, a fully and not partially observable model is considered. Typically, SOC models of human movements are partially observable since not all system states can be measured by the human and its measurements are subject to noise. Moreover, in \cite{Li.2011} only the cost function is determined. Although \cite{Schultheis.2021} propose an identification approach for the complete LQS control model, the covariance matrices are not estimated as well. Finally, \cite{Kolekar.2018} determine cost function and noise parameters of the LQS model but since the authors apply the model to describe human driving behavior, the inverse problem is not focused. Thus, a formal mathematical description of the problem is missing and only two cost and two noise parameters are estimated. Whereas this can be sufficient for the specific driving task, such few parameters fail to analyze the influence of all parameters available in a SOC model. Due to solely identifying the cost function and due to their missing consideration of the partially observable setting, the majority of known Inverse Optimal Control (IOC) methods are not applicable to the inverse problem of SOC models in their current form. They are used to identify the relative weights of different cost function candidates, penalizing e.g. jerk, torque or effort, for deterministic optimal control models of human motor planning. Bi-level optimization approaches can be found e.g. in \cite{Berret.2011,Albrecht.2012,Oguz.2018}, methods based on Karush-Kuhn-Tucker conditions in \cite{Lin.2016,Panchea.2018,Westermann.2020} and \cite{Jin.2019} utilize a method based on Hamilton optimality conditions. Such models assume a separation between motor planning and execution (see e.g. \cite{Flash.1985, Uno.1989} as seminal works). However, this traditional separation has been challenged in the last years \cite{Gallivan.2018} and SOC models arose.
\comment{IRL methods und auch Mainprice bewusst nicht ausführlich genannt} 

In summary, the inverse problem of LQ SOC models lacks formal definitions and no method exists that solves this Inverse Stochastic Optimal Control (ISOC) problem for the LQG and LQS case. Consequently, in this paper, we first define these ISOC problems and then, propose a new algorithm that solves them by iteratively estimating cost function and noise parameters via two bi-level optimizations. For their lower level we state new recursive calculations of the stochastic moments of the system state compared in the upper level with their ground truth values which leads to significantly reduced computation times. Moreover, our algorithm considers that in the ground truth data observed from the human not all system states are measured. Lastly, we provide simulation examples to show the effectiveness of our algorithm.
\tdd{
\begin{itemize}
	\item Abkürzungen LQG, ISOC einführen
\end{itemize}
}


\section{INVERSE STOCHASTIC OPTIMAL CONTROL PROBLEMS} \label{sec:isoc_problems}
In this section, formal definitions of the ISOC problem for the LQG (Subsection~\ref{subsec:isoc_problem_lqg}) and the LQS case (Subsection~\ref{subsec:isoc_problem_sensorimotor}) are given. Based on stating the respective SOC problems and the solutions to these forward optimal control problems, the inverse problems are defined.

\subsection{Linear-Quadratic Gaussian Case} \label{subsec:isoc_problem_lqg}
Let the discrete dynamics of a linear system be given by
\begin{align}
	\bm{x}_{t+1} &= \bm{A} \bm{x}_{t} + \bm{B} \bm{u}_{t} + \bm{\xi}_{t} \label{eq:LQGDynamics} \\
	\bm{y}_{t} &= \bm{H} \bm{x}_{t} + \bm{\omega}_{t} \label{eq:LQGFeedback},
\end{align}
where $\bm{x}\in\mathbb{R}^{n}$ denotes the system state (with additional time index $t$ the corresponding stochastic process $\bm{x}_t$ is described), $\bm{u}\in\mathbb{R}^{m}$ the control variables, $\bm{y}\in\mathbb{R}^{r}$ the observable output and $\bm{A}$, $\bm{B}$, $\bm{H}$ system matrices of appropriate dimensions that may depend on time. Furthermore, $\bm{\xi}_t \in \mathbb{R}^n$ and $\bm{\omega}_t \in \mathbb{R}^r$ are white Gaussian noise processes independent to each other and to $\bm{x}_t$. Each noise process $\bm{\xi}_t$ and $\bm{\omega}_t$ is composed by a standard white Gaussian noise process $\bm{\alpha}_t \in \mathbb{R}^p$ and $\bm{\beta}_t \in \mathbb{R}^q$ ($\cov{\bm{\alpha}_t}{\bm{\alpha}_t} = \covs{\bm{\alpha}_t} = \bm{I}$ and $\covs{\bm{\beta}_t}=\bm{I}$), respectively, where $\bm{\alpha}_t$ and $\bm{\beta}_t$ are independent to each other and to $\bm{x}_t$: $\bm{\xi}_t = \bm{\Sigma}^{\bm{\xi}} \bm{\alpha}_t$ and $\bm{\omega}_t = \bm{\Sigma}^{\bm{\omega}} \bm{\beta}_t$. Thus, the covariance matrices of $\bm{\xi}_t$ and $\bm{\omega}_t$ are defined by $\bm{\Omega}^{\bm{\xi}} = \covs{\bm{\xi}_t} = \bm{\Sigma}^{\bm{\xi}} \left(\bm{\Sigma}^{\bm{\xi}}\right)^{\mathrm{T}}$ and $\bm{\Omega}^{\bm{\omega}} = \covs{\bm{\omega}_t} = \bm{\Sigma}^{\bm{\omega}} \left(\bm{\Sigma}^{\bm{\omega}}\right)^{\mathrm{T}}$,
where $\bm{\Omega}^{\bm{\xi}}$ is assumed to be positive semi-definite and $\bm{\Omega}^{\bm{\omega}}$ positive definite. For the LQG case, we define the noise parameter vector $\bm{\sigma} \in \mathbb{R}^{\Sigma}$ as $\bm{\sigma} = \mat{\vecv(\bm{\Sigma}^{\bm{\xi}})^{\mathrm{T}} & \vecv(\bm{\Sigma}^{\bm{\omega}})^{\mathrm{T}}}^{\mathrm{T}}$. 
Finally, the initial values of $\bm{x}_t$ are given by $\E{\bm{x}_{0}}$ and $\bm{\Omega}^{\bm{x}}_{0} = \covs{\bm{x}_0}$ ($\bm{\Omega}^{\bm{x}}_{0}$ positive semi-definite). 
\comment{Bedingungen an die Kovarianzmatrizen können durch geeignete Wahl der $\sigma_i$ erreicht werden, durch bekannte Beziehung zu den Kovarianzmatrizen folgen aus $\bm{\sigma}$ die Kovarianzmatrizen $\bm{\Omega}^{\bm{\xi}}$ und $\bm{\Omega}^{\bm{\omega}}$, z.B. werden Bedingungen an Kovarianzmatrizen bei diagonalen $\bm{\Sigma}^{\bm{\xi}}$ und $\bm{\Sigma}^{\bm{\omega}}$ durch rein positive Werte $\sigma_i$ erfüllt - Bedingungen an Matrizen und/oder $\sigma_i$ können in GS des upper level berücksichtigt werden ($J_{GS} = - \infty$ bei nicht geeigneten grid points)}
\comment{$\bm{\sigma}$ wird hier stets über die vektorisierten Matrizen bestimmt, damit kann dann sowohl $\bm{\sigma}^{\ast}$ als auch $\tilde{\bm{\sigma}}$ einfach definiert werden; am Ende wird dann lediglich über die als bekannt angenommenen nicht in $\bm{\sigma}^*$ verschwindenden Elemente optimiert, um $\tilde{\bm{\sigma}}$ zu erreichen}

Considering the LQG optimal control problem, the aim is to find an admissible control strategy for \eqref{eq:LQGDynamics} and \eqref{eq:LQGFeedback} that minimizes a performance criterion \cite[p.~258]{Astrom.1970}, which is defined by
\begin{align}
	J = \E{\bm{x}^{\mathrm{T}}_N \bm{Q}_N \bm{x}_N + \sum_{t=0}^{N-1} \bm{x}^{\mathrm{T}}_t \bm{Q} \bm{x}_t + \bm{u}^{\mathrm{T}}_t \bm{R} \bm{u}_t}, \label{eq:LQGCost}
\end{align} 
where $\bm{Q}_N$, $\bm{Q}$ and $\bm{R}$ are symmetric matrices of appropriate dimensions with $\bm{Q}_N$, $\bm{Q}$ positive semi-definite and $\bm{R}$ positive definite. For these cost function matrices, we introduce the notations\footnote{This notation is comparable to describing the cost function as linear combination of basis functions, common in the IOC literature (see e.g. \cite{Molloy.2020}), since e.g. each $\bm{x}^{\mathrm{T}}\bm{q}_{Q,i}\bm{q}_{Q,i}^{\mathrm{T}}\bm{x}$ represents such a basis function.} $\bm{Q}_N = \sum_{i = 1}^{S_N} s_{N,i} \bm{q}_{N,i} \bm{q}_{N,i}^{\mathrm{T}}$, $\bm{Q} = \sum_{i = 1}^{S_Q} s_{Q,i} \bm{q}_{Q,i} \bm{q}_{Q,i}^{\mathrm{T}}$ and $\bm{R} = \sum_{i = 1}^{S_R} s_{R,i} \bm{q}_{R,i} \bm{q}_{R,i}^{\mathrm{T}}$,
where $\bm{q}_{N,i}, \bm{q}_{Q,i} \in \mathbb{R}^{n}$ and $\bm{q}_{R,i} \in \mathbb{R}^{m}$. 
For the later inverse problem definition, we define the cost function parameter vector $\bm{s} \in \mathbb{R}^{S}$ ($S = S_N + S_Q + S_R$) as $\bm{s}^{\mathrm{T}} = \mat{s_{N,1} & \!\!\dots\! & s_{N,S_N} & s_{Q,1} & \!\dots\! & s_{Q,S_Q} & s_{R,1} & \!\dots\! & s_{R,S_R}}$. 
\comment{auch hier kann durch die geeignete Wahl der $s_i$ stets ein wohldefiniertes LQG-Problem gewährleistet werden, sprich die Bedingungen an die Kostenfunktionsmatrizen erfüllt werden, z.B. bei Einheitsvektoren für die $\bm{q}_i$ (Diagonalmatrizen entstehen) und positiven Werten für die $s_i$ - Bedingungen an Matrizen und/oder $s_i$ in upper level optimization durch Prüfung der grid points möglich, analog Kovarianzmatrizen und $\sigma_i$}
\comment{um im späteren Simulationsbeispiel eine Bestrafung der Abweichungen von $\bm{x}_{\text{ref}}$ zu ermöglichen, wird später eine Erweiterung des Zustands eingeführt, die die nötigen Sollverläufe enthält - dabei werden die $\bm{q}_{N,i}$ gerade geeignet gewählt, um die Bestrafung der Abweichungen zu realisieren - die Idee entstammt der Realisierung des OC Tracking durch die Erweiterung des Systems um ein Exosystem (vgl. Notizen IOC Tracking) - jedoch muss im hier vorliegenden Fall auf partially observable model geachtet werden, sprich damit das Regelgesetz die Aufteilung in feedback und feedforward controller korrekt durchführt, müssen in $\hat{\bm{x}}_t$ immer die tatsächlichen Sollgrößen enthalten sein (Mensch kennt diese ja auch), durch geschlossene Lösung in den Lemmata kann gerade nicht "Überschreibung" in jedem Zeitschritt erfolgen - im Fall von $\bm{x}_{\text{ref}}$: konstante Referenz entsteht, Initialisierung der erweiterten Zustände des Systems und Schätzers mit $\bm{x}_{\text{ref}}$, keine Varianz in diesen Größen, zzgl. keine Verrauschung in diesen Größen, Annahme Messbarkeit und keine Verrauschung in den entsprechenden Ausgangsgrößen (in Sim ergibt sich damit exakt gewünschtes Verhalten, sprich keine Varianz in den erweiterten Zuständen und Mittelwert stets auf dem Referenzwert) - für ISOC entsteht damit formal immer Fall der ground truth data mit nicht vollständig messbaren Zuständen - im Fall von Tracking $\bm{x}_{t,\text{ref}}$ (Tracking-Fall aber für vorliegendes Paper out-of-the-scope, daher wird hier auf zeitvariante Gewichtungsvektoren verzichtet): muss auch über konstante erweiterte Zustandsgröße realisiert werden, da Vorsteuerungssignal a priori nicht berechnet werden kann, hier wird $\bm{x}_{t,\text{ref}}$ in zeitvariante (bekannte) Basisvektoren $\bm{q}_{Q,t,i}$ integriert, womit letzte Zustandsgröße des erweiterten Systems (die Referenz widerspiegelt) zu konstanter Größe ($= 1$) wird}

According to \cite[p.~274]{Astrom.1970}, the LQG problem is solved by the control law $\bm{u}_{t} = -\bm{L}_{t} \hat{\bm{x}}_{t}$ with
\begin{align}
 	\bm{L}_{t} &= \left(\bm{R} + \bm{B}^{\mathrm{T}} \bm{Z}_{t+1} \bm{B} \right)^{-1} \bm{B}^{\mathrm{T}} \bm{Z}_{t+1} \bm{A} \label{eq:LQGControlLawL}
\end{align}
and $\hat{\bm{x}}_t$ denoting the estimation of $\bm{x}_t$ calculated by $\hat{\bm{x}}_{t+1} = \bm{A} \hat{\bm{x}}_{t} + \bm{B} \bm{u}_{t} + \bm{K}_{t} \left( \bm{y}_{t} - \bm{H} \hat{\bm{x}}_{t} \right)$, where
\begin{align}
	\bm{K}_{t} &= \bm{A} \bm{P}_{t} \bm{H}^{\mathrm{T}} \left( \bm{H} \bm{P}_{t} \bm{H}^{\mathrm{T}} + \bm{\Omega}^{\bm{\omega}} \right)^{-1} \label{eq:LQGEstimatorK}.
\end{align}
Here, $\bm{Z}_t$ and $\bm{P}_t$ are computed via Riccati difference equations with boundary values $\bm{Z}_{N} = \bm{Q}_{N}$ and $\bm{P}_{0} = \bm{\Omega}_{0}^{\bm{x}}$ \cite[p.~229, p.~274]{Astrom.1970}.

As explained in Section~\ref{sec:intro}, given the control $\bm{L}_t$ and filter matrices $\bm{K}_t$, we need a recursive calculation of $\E{\bm{x}_t}$ and $\bm{\Omega}_t^{\bm{x}} = \covs{\bm{x}_t}$ for an efficient realization of the lower level of the bi-level optimizations in our ISOC algorithm later.
\comment{kein conditional mean, da Gesamterwartungswert und -varianz zu vgl. bei bi-level ISOC - wird gemessen}
\comment{bei LQG könnte auch für $\E{\bm{x}_t}$ (für $\bm{\Omega}_t^{\bm{x}}$ und $\bm{P}_t$ nicht, da problematisch wg. nochmaliger Erwartungswertbildung, auch $\bm{P}_t$ aus Rechnungen MA Simon nicht ohne Weiteres rekonstruierbar) direkt $\hat{\bm{x}}_t$ mit nochmaliger Erwartungswertbildung über $\bm{y}_t$ genutzt werden (wenn eine Initialisierung des KF mit den initial values of $\bm{x}_t$ durchgeführt wird (im nachfolgenden Lemma aber nach MA Simon, einfacher) - geht bei uns, da als bekannt angenommen - KF ist dann erwartungstreu), im Beweis müsste dann kurz auf Theorem~4.1 aus \cite{Astrom.1970} verwiesen werden - daraus folgt, dass Erwartungswert unter Bedingung der Messungen von $\bm{y}_t$ bis zum aktuellen Zeitpunkt (deshalb nochmalige Erwartungswertbildung über $\bm{y}_t$) und Kovarianzmatrix direkt aus den Schätzergleichungen folgen (beide Größen sind sufficient statistic for conditional distribution unter Bedingung der Messungen von $\bm{y}_t$ bis zum aktuellen Zeitpunkt), wobei der Schätzer optimal bzgl. der Minimierung des quadrierten Schätzfehlers ist - dieser Zusammenhang zur sufficient statistic der conditional distribution unter Bedingung der Messungen von $\bm{y}_t$ bis zum aktuellen Zeitpunkt gilt aber gerade nur bei LQG, bei sensorimotor case wird zwar auch Minimierung quadrierter Schätzfehler durchgeführt (über gemeinsame Kostenfunktion für controller und estimator) allerdings fehlt formale Aussage zu dieser conditional distribution und auch separation theorem, dort kann lediglich die Erwartungstreue des als linear angenommenen Filters gezeigt werden unter Bedingung von $\hat{\bm{x}}_t$ (womit theoretisch auch dort $\E{\bm{x}_t}$ aus den Schätzergleichungen über eine nochmalige Erwartungswertbildung über $\hat{\bm{x}}_t$ abgeleitet werden könnte)}
\begin{lemma} \label{lemma:LQGSolution} 
	Let the LQG control problem be defined by \eqref{eq:LQGDynamics}, \eqref{eq:LQGFeedback} and \eqref{eq:LQGCost}. Let the solution be given by the control $\bm{L}_t$~\eqref{eq:LQGControlLawL} and filter matrices $\bm{K}_t$~\eqref{eq:LQGEstimatorK}. Then, the mean $\E{\bm{x}_t}$ and covariance $\bm{\Omega}_t^{\bm{x}}$ of $\bm{x}_t$ are computed by
	\begin{align}
		\mat{\E{\bm{x}_{t+1}} \\ \E{\hat{\bm{x}}_{t+1}}} &= \bm{\mathcal{A}}_t \mat{\E{\bm{x}_{t}} \\ \E{\hat{\bm{x}}_{t}}} \label{eq:LQGSolutionEW}, \\
		\mat{\bm{\Omega}_{t+1}^{\bm{x}} \!\!& \bm{\Omega}_{t+1}^{\bm{x}\hat{\bm{x}}} \\ \bm{\Omega}_{t+1}^{\hat{\bm{x}}\bm{x}} \!\!& \bm{\Omega}_{t+1}^{\hat{\bm{x}}}} &= \bm{\mathcal{A}}_t \mat{\bm{\Omega}_{t}^{\bm{x}} \!\!\!\!& \bm{\Omega}_{t}^{\bm{x}\hat{\bm{x}}} \\ \bm{\Omega}_{t}^{\hat{\bm{x}}\bm{x}} \!\!\!\!& \bm{\Omega}_{t}^{\hat{\bm{x}}}} \bm{\mathcal{A}}^\mathrm{T}_{t} + \mat{\bm{\Omega}^{\bm{\xi}} \!\!\!\!& \bm{0} \\ \bm{0} \!\!\!\!& \bm{K}_{t} \bm{\Omega}^{\bm{\omega}} \bm{K}^{\mathrm{T}}_{t}} \label{eq:LQGSolutionCOV}
	\end{align}
	with $\bm{\Omega}_{t}^{\bm{x}\hat{\bm{x}}} = \cov{\bm{x}_t}{\hat{\bm{x}}_t}$, $\bm{\Omega}_{t}^{\hat{\bm{x}}\bm{x}} = \cov{\hat{\bm{x}}_t}{\bm{x}_t}$, $\bm{\Omega}_{t}^{\hat{\bm{x}}} = \covs{\hat{\bm{x}}_t}$,
	\begin{equation}
		\bm{\mathcal{A}}_{t} = \begin{bmatrix} \bm{A} & -\bm{B} \bm{L}_{t} \\ \bm{K}_{t} \bm{H} & \bm{A} - \bm{K}_{t} \bm{H} - \bm{B} \bm{L}_{t} \end{bmatrix} \label{eq:MathcalA}
	\end{equation}
	and initial values $\E{\hat{\bm{x}}_0}=\hat{\bm{x}}_0=\E{\bm{x}_0}$ and
	\begin{equation}
		\mat{\bm{\Omega}_{0}^{\bm{x}} \!\!\!\!& \bm{\Omega}_{0}^{\bm{x}\hat{\bm{x}}} \\ \bm{\Omega}_{0}^{\hat{\bm{x}}\bm{x}} \!\!\!\!& \bm{\Omega}_{0}^{\hat{\bm{x}}}} = \mat{\bm{\Omega}_{0}^{\bm{x}} & \bm{0} \\ \bm{0} & \bm{0}} \label{eq:InitialOmega}.
	\end{equation}
\end{lemma}
\begin{proof}
	With $\bm{u}_t = -\bm{L}_t \hat{\bm{x}}_t$ in \eqref{eq:LQGDynamics} and in the filter equation for $\hat{\bm{x}}_{t+1}$, \eqref{eq:LQGSolutionEW} results. Using the same expressions for $\bm{x}_{t+1}$ and $\hat{\bm{x}}_{t+1}$ in
	\begin{equation}
		\mat{\bm{\Omega}_{t+1}^{\bm{x}} \!\!& \bm{\Omega}_{t+1}^{\bm{x}\hat{\bm{x}}} \\ \bm{\Omega}_{t+1}^{\hat{\bm{x}}\bm{x}} \!\!& \bm{\Omega}_{t+1}^{\hat{\bm{x}}}} = \covs{\mat{\bm{x}_{t+1} \\ \hat{\bm{x}}_{t+1}}}
	\end{equation}
	\eqref{eq:LQGSolutionCOV} is obtained by considering the independence of the noise processes $\bm{\xi}_t$ and $\bm{\omega}_t$ between each other and to $\bm{x}_t$ as well as $\hat{\bm{x}}_t$. The initial values follow from the initial value $\bm{\Omega}_{0}^{\bm{x}}$ of $\bm{x}_t$ and the initialization of the filter equations with $\hat{\bm{x}}_0=\E{\bm{x}_0}$.
\end{proof}

Lemma~\ref{lemma:LQGSolution} yields that the average behavior $\E{\bm{x}_t}$ predicted by the LQG model is solely influenced by the cost function parameters $\bm{s}$, which can be seen by \eqref{eq:LQGSolutionEW} and \eqref{eq:LQGControlLawL}, since $\bm{L}_t$ is independent of the covariance matrices. This is a special property of the LQG model due to the separation theorem. In the next subsection, we show that in the LQS case the covariance matrices are necessary to compute the average behavior as well.

Now, we can define the ISOC problem for the LQG case.
\comment{Problem bewusst nur mit Erwartungswert- und Varianzmatching formulieren, darauf fittet am Ende auch ISOC algorithm (und aufgestellte Lemmata berechnen diese entsprechend) - theoretisch könnte Problem auch allgemeiner definiert werden, mit Blick auf mehrere/höhere stochastische Momente oder sogar Wahrscheinlichkeitsverteilung (sprich dann alle Momente) - aber hier out-of-the-scope, auch weil i.Allg. nur bei LQG Erwartungswert und Varianz hinreichend für Beschreibung der Wahrscheinlichkeitsverteilung sind (vgl. multiplikative/nichtlineare Zusammenhänge zwischen Rauschprozessen beim sensorimotor case), daher könnte auf das matching höherer Momente/der Wahrscheinlichkeitsverteilung erstmal nur im LQG-Fall geschlossen werden - im ersten Schritt reicht damit erstmal Mittelwert und Varianz aus und das sind auch die zentralen Größen zur Beschreibung des mittleren Verhaltens und der Variabilität menschlicher Bewegungen in den ersten Untersuchungen}

\begin{assumption} \label{assumption:LQGTrajectorySamples}
	Suppose a set $\{{\bm{M}\bm{x}^{\ast}_{t}}^{(k)}\}$ ($k \in \{1,\dots,K\}$) of time-discrete trajectories ${\bm{M}\bm{x}^{\ast}_{t}}^{(k)}$ is given, where ${\bm{x}^{\ast}_{t}}^{(k)}$ denotes a realization of the stochastic process $\bm{x}_t^{\ast}$ and $\bm{M} \in \mathbb{R}^{\bar{n} \times n}$ follows from the identity matrix $\bm{I}$ by deleting rows corresponding to states which are not measured in the ground truth data. The stochastic process $\bm{x}_t^{\ast}$ results from the estimation-control loop consisting of \eqref{eq:LQGDynamics} and \eqref{eq:LQGFeedback} with $\bm{L}_t^{\ast}$~\eqref{eq:LQGControlLawL} and $\bm{K}_t^{\ast}$~\eqref{eq:LQGEstimatorK} corresponding to unknown cost function $\bm{s}^{\ast}$ and noise parameters $\bm{\sigma}^{\ast}$. Furthermore, let $\hat{\bm{m}}_t \approx \E{\bm{M}\bm{x}_t^{\ast}}$ and $\hat{\bm{\Omega}}^{\bm{x}^{\ast}}_{t} \approx \bm{M}\bm{\Omega}^{\bm{x}^{\ast}}_{t}\bm{M}^{\mathrm{T}}$ be estimates of the mean and covariance of the measured states of $\bm{x}_t^{\ast}$ computed from the set of trajectories.
\end{assumption}
\begin{assumption} \label{assumption:LQGBasisVectorsNonZeroElements}
	It is known which elements of $\bm{\sigma}^{\ast}$ are non-zero, but not their exact numerical values. Furthermore, the basis vectors $\bm{q}_{N,i}$, $\bm{q}_{Q,i}$ and $\bm{q}_{R,i}$ corresponding to the non-zero coefficients in $\bm{s}^{\ast}$ are known. Moreover, the system matrices $\bm{A}$, $\bm{B}$ and $\bm{H}$ are known.
\end{assumption}

\begin{problem} \label{problem:ISOCLQG}
	Let Assumptions~\ref{assumption:LQGTrajectorySamples} and \ref{assumption:LQGBasisVectorsNonZeroElements} hold. Find parameters $\tilde{\bm{s}}$ and $\tilde{\bm{\sigma}}$ such that $\tilde{\bm{L}}_t$~\eqref{eq:LQGControlLawL} and $\tilde{\bm{K}}_t$~\eqref{eq:LQGEstimatorK} resulting from $\tilde{\bm{s}}$ and $\tilde{\bm{\sigma}}$ lead to $\tilde{\bm{x}}_t$ in the estimation-control loop (\eqref{eq:LQGDynamics} and \eqref{eq:LQGFeedback} with $\tilde{\bm{L}}_t$ and $\tilde{\bm{K}}_t$) with $\E{\bm{M}\tilde{\bm{x}}_t}=\E{\bm{M}\bm{x}^{\ast}_t}$ and $\bm{M}\bm{\Omega}^{\tilde{\bm{x}}}_{t}\bm{M}^{\mathrm{T}}=\bm{M}\bm{\Omega}^{\bm{x}^{\ast}}_{t}\bm{M}^{\mathrm{T}}$.
\end{problem}

\begin{remark} \label{remark:DefinitionReference}
	Assumption~\ref{assumption:LQGTrajectorySamples} defines the ground truth data for an ISOC algorithm solving Problem~\ref{problem:ISOCLQG}. In a practical setup, these ground truth data is observed from the human. In Assumption~\ref{assumption:LQGTrajectorySamples}, $\bm{M}$ takes into account that not all system states can be measured in general. In addition, in several models of human biomechanics (see e.g. Section~\ref{sec:sim_results}) the human control signal as well as the perceived output are only virtual quantities and thus not measured when a human is observed. Hence, we do not consider measurements of realizations of $\bm{u}^{\ast}_t$ or $\bm{y}^{\ast}_t$ and do not aim at finding parameters $\tilde{\bm{s}}$ and $\tilde{\bm{\sigma}}$ yielding matching mean and covariance of the control or output variables as well.
\end{remark}

\subsection{Linear-Quadratic Sensorimotor Case} \label{subsec:isoc_problem_sensorimotor}
The LQG model cannot fully account for the characteristic variability patterns of human movements \cite{Todorov.2002,Kolekar.2018}. Hereto, a control-dependent noise process $\sum_{i=1}^{c} \varepsilon_t^{(i)} \bm{C}_i \bm{u}_t$ considering that higher control magnitudes lead to higher motor noise and a state-dependent noise process $\sum_{i=1}^{d} \epsilon_t^{(i)} \bm{D}_i \bm{x}_t$ accounting for multiplicative noise in human perception need to be added to \eqref{eq:LQGDynamics} and \eqref{eq:LQGFeedback} \cite{Todorov.2005}:
\begin{align}
	\bm{x}_{t+1} &= \bm{A} \bm{x}_{t} + \bm{B} \bm{u}_{t} + \bm{\xi}_{t} + \sum_{i=1}^{c} \varepsilon_t^{(i)} \bm{C}_i \bm{u}_t \label{eq:SensoDynamics} \\
	\bm{y}_{t} &= \bm{H} \bm{x}_{t} + \bm{\omega}_{t}  + \sum_{i=1}^{d} \epsilon_t^{(i)} \bm{D}_i \bm{x}_t \label{eq:SensoFeedback},
\end{align}
where $\bm{x}$, $\bm{u}$, $\bm{y}$, $\bm{A}$, $\bm{B}$, $\bm{H}$, $\bm{\xi}_{t}$ and $\bm{\omega}_t$ are defined as in Subsection~\ref{subsec:isoc_problem_lqg}. 
The time-independent matrices $\bm{C}_i$ and $\bm{D}_i$ are composed by $\bm{C}_i = \sigma_{i}^{\bm{u}}\bm{B}\bm{F}_i$ and $\bm{D}_i = \sigma_{i}^{\bm{x}}\bm{H}\bm{G}_i$, respectively, where $\bm{C}_i$, $\bm{D}_i$, $\bm{F}_i$ and $\bm{G}_i$ have appropriate dimensions. Furthermore, $\bm{\varepsilon}_t = \mat{\varepsilon_t^{(1)} & \dots & \varepsilon_t^{(c)}}^{\mathrm{T}}$ and $\bm{\epsilon}_t = \mat{\epsilon_t^{(1)} & \dots & \epsilon_t^{(d)}}^{\mathrm{T}}$ are standard white Gaussian noise processes ($\covs{\bm{\varepsilon}_t} = \covs{\bm{\epsilon}_t} = \bm{I}$) independent to each other, $\bm{\xi}_t$, $\bm{\omega}_t$ and $\bm{x}_t$. For the LQS case, we define the noise parameter vector $\bm{\sigma} \in \mathbb{R}^{\Sigma}$ as $\bm{\sigma} = \mat{\vecv(\bm{\Sigma}^{\bm{\xi}})^{\mathrm{T}} & \vecv(\bm{\Sigma}^{\bm{\omega}})^{\mathrm{T}} & \sigma_{1}^{\bm{u}} & \dots & \sigma_{c}^{\bm{u}} & \sigma_{1}^{\bm{x}} & \dots & \sigma_{d}^{\bm{x}}}^{\mathrm{T}}$. 
\comment{durch hier vorgenommene Definition der Skalierungsmatrizen $\bm{C}_i$, $\bm{D}_i$ wird sowohl bei den additiven als auch den multiplikativen Rauschprozessen immer der Skalierungsfaktor einer skalaren standardnormalverteilten Zufallsgröße bestimmt (und in $\bm{\sigma}$ aufgenommen) - als bekannt (oder postuliert) angenommene Struktur, wie dann Wirkung auf $\bm{x}$ oder $\bm{y}$ erfolgt entspricht sozusagen der aus IOC/IRL bekannten Kenntnis der Basisfunktionen (oder der Kenntnis der Basisvektoren bei den hier genutzten Gütemaßen) - resultierende Kovarianzmatrizen der multiplikativen Rauschprozesse sind per Definition positiv semi-definit und symmetrisch, keine weitere Überprüfung diesbezüglich nötig}

According to \cite{Todorov.2005}, the separation theorem no longer holds and one can compute an approximate (suboptimal) solution to the forward problem (minimization of \eqref{eq:LQGCost} subject to \eqref{eq:SensoDynamics} and \eqref{eq:SensoFeedback}) consisting of the control law $\bm{u}_{t} = -\bm{L}_{t} \hat{\bm{x}}_{t}$ and estimator law $\hat{\bm{x}}_{t+1} = \bm{A}\hat{\bm{x}}_{t} + \bm{B}\bm{u}_t + \bm{K}_{t} \left( \bm{y}_{t} - \bm{H} \hat{\bm{x}}_{t} \right) + \bm{\eta}_t$ by iterating between
\begin{align}
	\bm{L}_{t} &= \Big(\bm{R} + \bm{B}^{\mathrm{T}} \bm{Z}_{t+1}^{\bm{x}} \bm{B} \nonumber\\
	&\hphantom{=} + \sum_{i} \bm{C}_i^{\mathrm{T}} \left( \bm{Z}_{t+1}^{\bm{x}} + \bm{Z}_{t+1}^{\bm{e}} \right) \bm{C}_i \Big)^{-1} \bm{B}^{\mathrm{T}} \bm{Z}_{t+1}^{\bm{x}} \bm{A} \label{eq:SensoControlLawL} 
\end{align}
and
\begin{align}
	\bm{K}_{t} &= \bm{A} \bm{P}_{t}^{\bm{e}} \bm{H}^{\mathrm{T}} \Big( \bm{H} \bm{P}_{t}^{\bm{e}} \bm{H}^{\mathrm{T}} + \bm{\Omega}^{\bm{\omega}} \nonumber \\
	&\hphantom{=} + \sum_{i} \bm{D}_i \left( \bm{P}_t^{\bm{e}} + \bm{P}_t^{\hat{\bm{x}}} + \bm{P}_t^{\hat{\bm{x}}\bm{e}} + \bm{P}_t^{\bm{e}\hat{\bm{x}}} \right) \bm{D}_i^{\mathrm{T}} \Big)^{-1} \label{eq:SensoEstimatorK}
\end{align}
starting with an initial estimate for $\bm{K}_t$. The recursive equations to compute $\bm{Z}_{t}^{\bm{x}}$, $\bm{Z}_{t}^{\bm{e}}$, $\bm{P}_{t}^{\bm{e}}$, $\bm{P}_{t}^{\hat{\bm{x}}}$, $\bm{P}_{t}^{\hat{\bm{x}}\bm{e}}$ and $\bm{P}_{t}^{\bm{e}\hat{\bm{x}}}$ can be found in \cite{Todorov.2005}. The white Gaussian noise process $\bm{\eta}_t \in \mathbb{R}^{n}$ has the symmetric and positive semi-definite covariance matrix $\bm{\Omega}^{\bm{\eta}}$ and is independent to all other stochastic processes. This noise process accounts for errors of the internal model used by the human for signal processing \cite{Kolekar.2018}. However, since its characteristics are unexplored until now, we omit $\bm{\eta}_t$ in our ISOC algorithm and do not include the parameters of $\bm{\Omega}^{\bm{\eta}}$ in $\bm{\sigma}$. This would be possible by following the procedure of $\bm{\xi}_t$ and $\bm{\omega}_t$ for $\bm{\eta}_t$. 
\comment{bei Lösung kann hier nur nicht-adaptiver Filter genutzt werden, da später mit Lemma~\ref{lemma:SensoSolution} gerade Berechnung von Erwartungswert und Kovarianz mit $\bm{K}_t$ und $\bm{L}_t$ durchgeführt wird (für gesamten Zeithorizont) - bei adaptivem Filter ändert sich aber Filtermatrix in Abhängigkeit vom konkret auftretenden $\bm{x}$-Wert in einem Durchlauf des estimation-control loop, daher könnte in dem Fall eine Bestimmung von Erwartungswert und Kovarianz wieder nur über Einzelsimulationen erfolgen}

Again, we state a recursive calculation of $\E{\bm{x}_t}$ and $\bm{\Omega}_t^{\bm{x}}$ for our ISOC algorithm later.
\begin{lemma} \label{lemma:SensoSolution} 
	Let the LQS control problem be defined by \eqref{eq:SensoDynamics}, \eqref{eq:SensoFeedback} and \eqref{eq:LQGCost}. Let the (approximate) solution be given by $\bm{L}_t$~\eqref{eq:SensoControlLawL} and $\bm{K}_t$~\eqref{eq:SensoEstimatorK}. Then, the mean $\E{\bm{x}_t}$ and covariance $\bm{\Omega}_t^{\bm{x}}$ of $\bm{x}_t$ are computed by 
	\begin{align}
		\mat{\E{\bm{x}_{t+1}} \\ \E{\hat{\bm{x}}_{t+1}}} &= \bm{\mathcal{A}}_t \mat{\E{\bm{x}_{t}} \\ \E{\hat{\bm{x}}_{t}}} \label{eq:SensoSolutionEW}, \\
		\mat{\bm{\Omega}_{t+1}^{\bm{x}} \!\!& \bm{\Omega}_{t+1}^{\bm{x}\hat{\bm{x}}} \\ \bm{\Omega}_{t+1}^{\hat{\bm{x}}\bm{x}} \!\!& \bm{\Omega}_{t+1}^{\hat{\bm{x}}}} &= \bm{\mathcal{A}}_t \mat{\bm{\Omega}_{t}^{\bm{x}} \!\!\!\!& \bm{\Omega}_{t}^{\bm{x}\hat{\bm{x}}} \\ \bm{\Omega}_{t}^{\hat{\bm{x}}\bm{x}} \!\!\!\!& \bm{\Omega}_{t}^{\hat{\bm{x}}}} \bm{\mathcal{A}}^\mathrm{T}_{t} \nonumber \\		
		&\hphantom{=} + \mat{\bm{\Omega}^{\bm{\xi}} \!\!\!\!& \bm{0} \\ \bm{0} \!\!\!\!& \bm{\Omega}^{\bm{\eta}} + \bm{K}_{t} \bm{\Omega}^{\bm{\omega}} \bm{K}^{\mathrm{T}}_{t}} + \mat{\bar{\bm{\Omega}}_t^{\hat{\bm{x}}} \!\!\!\!& \bm{0} \\ \bm{0} \!\!\!\!& \bar{\bm{\Omega}}_t^{\bm{x}}} \label{eq:SensoSolutionCOV}
	\end{align}
	with $\bar{\bm{\Omega}}_t^{\hat{\bm{x}}} = \sum_{i} \bm{C}_i \bm{L}_t \left( \bm{\Omega}_{t}^{\hat{\bm{x}}} + \E{\hat{\bm{x}}_t}\E{\hat{\bm{x}}_t}^{\mathrm{T}} \right) \bm{L}_t^{\mathrm{T}} \bm{C}_i^{\mathrm{T}}$, $\bar{\bm{\Omega}}_t^{\bm{x}} = \sum_{i} \bm{K}_t \bm{D}_i \left( \bm{\Omega}_{t}^{\bm{x}} + \E{\bm{x}_t}\E{\bm{x}_t}^{\mathrm{T}} \right) \bm{D}_i^{\mathrm{T}} \bm{K}_t^{\mathrm{T}}$, $\bm{\mathcal{A}}_t$~\eqref{eq:MathcalA} and initial values $\E{\hat{\bm{x}}_0}=\hat{\bm{x}}_0=\E{\bm{x}_0}$ and \eqref{eq:InitialOmega}.
\end{lemma}
\begin{proof}
	See Appendix.
\end{proof}

Lemma~\ref{lemma:SensoSolution} yields that the average behavior $\E{\bm{x}_t}$ predicted by the LQS model depends on the covariance matrices and thus noise parameters $\bm{\sigma}$ since the calculation of $\bm{L}_t$~\eqref{eq:SensoControlLawL} depends on the $\bm{C}_i$ and $\bm{K}_t$ which in turn depend on $\bm{\sigma}$. Hence, the identification of the noise parameters are not only necessary for describing the variability patterns of human movements but also its average behavior.

Now, the ISOC problem of the LQS case can be defined.
\begin{assumption} \label{assumption:SensoTrajectorySamples}
	Assumption~\ref{assumption:LQGTrajectorySamples} holds with $\bm{x}^{\ast}_t$ being the stochastic process of the system state in the estimation-control loop consisting of \eqref{eq:SensoDynamics} and \eqref{eq:SensoFeedback} with $\bm{L}^{\ast}_t$~\eqref{eq:SensoControlLawL} and $\bm{K}^{\ast}_t$~\eqref{eq:SensoEstimatorK} resulting from unknown cost function $\bm{s}^{\ast}$ and noise parameters $\bm{\sigma}^{\ast}$.
\end{assumption}
\begin{assumption} \label{assumption:SensoBasisVectorsNonZeroElements}
	Assumption~\ref{assumption:LQGBasisVectorsNonZeroElements} holds. Furthermore, the $c$ scaling matrices $\bm{F}_i$ and the $d$ scaling matrices $\bm{G}_i$ are known.
\end{assumption}

\begin{problem} \label{problem:ISOCSenso}
	Let Assumptions~\ref{assumption:SensoTrajectorySamples} and \ref{assumption:SensoBasisVectorsNonZeroElements} hold. Find parameters $\tilde{\bm{s}}$ and $\tilde{\bm{\sigma}}$ such that $\tilde{\bm{L}}_t$~\eqref{eq:SensoControlLawL} and $\tilde{\bm{K}}_t$~\eqref{eq:SensoEstimatorK} resulting from $\tilde{\bm{s}}$ and $\tilde{\bm{\sigma}}$ lead to $\tilde{\bm{x}}_t$ in the estimation-control loop (\eqref{eq:SensoDynamics} and \eqref{eq:SensoFeedback} with $\tilde{\bm{L}}_t$ and $\tilde{\bm{K}}_t$) with $\E{\bm{M}\tilde{\bm{x}}_t}=\E{\bm{M}\bm{x}^{\ast}_t}$ and $\bm{M}\bm{\Omega}^{\tilde{\bm{x}}}_{t}\bm{M}^{\mathrm{T}}=\bm{M}\bm{\Omega}^{\bm{x}^{\ast}}_{t}\bm{M}^{\mathrm{T}}$.
\end{problem}



\section{SOLVING ISOC PROBLEMS: A BI-LEVEL-BASED APPROACH} \label{sec:solving_isoc}
In order to solve both ISOC problems stated in the previous section, we propose a new algorithm in the following. Before its detailed definition is given in Subsection~\ref{subsec:gridsearch_isoc}, we describe the general approach in Subsection~\ref{subsec:general_isoc_approach}.

\subsection{General Approach} \label{subsec:general_isoc_approach}
\begin{figure}[t]
	\centering
	\includegraphics[width=3.3in]{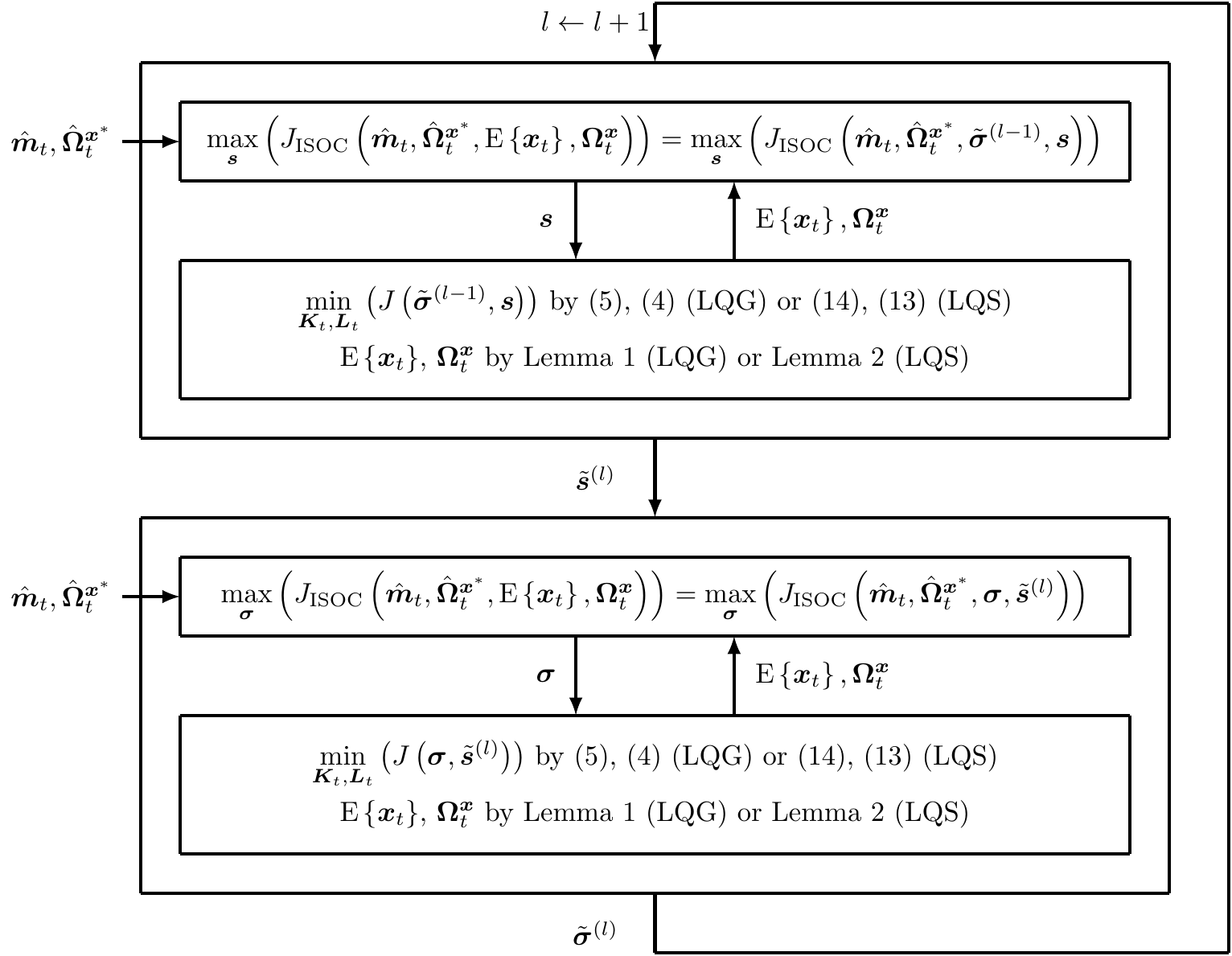}
	\caption{Overview of our ISOC algorithm.}
	\label{fig:general_algorithm}
\end{figure}
In the following, we explain the general procedure of our ISOC algorithm shown in Fig.~\ref{fig:general_algorithm}. Starting with a set of ground truth trajectories according to Assumptions~\ref{assumption:LQGTrajectorySamples} and \ref{assumption:SensoTrajectorySamples}, we state a direct optimization problem for finding parameters $\tilde{\bm{s}}$ and $\tilde{\bm{\sigma}}$. Hereto, we introduce a performance criterion $J_{\ISOC}$ describing how well the mean $\E{\bm{x}_t}$ and covariance value $\bm{\Omega}^{\bm{x}}_t$ of a current guess of $\bm{s}$ and $\bm{\sigma}$ match $\hat{\bm{m}}_t$ and $\hat{\bm{\Omega}}^{\bm{x}^{\ast}}_t$. In our case, we propose $J_{\ISOC}$ as follows:
\begin{equation}
	J_{\ISOC} = \frac{\bm{w}_{\cdot, \text{m}}^{\mathrm{T}} \bm{m}^{\text{VAF}} + \bm{w}_{\cdot, \text{v}}^{\mathrm{T}} \vecv\left(\bm{\Omega}^{\text{VAF}} \right)}{\norm{\bm{w}_{\cdot, \text{m}}}_1 + \norm{\bm{w}_{\cdot, \text{v}}}_1} \label{eq:J_GS},
\end{equation}
where $\bm{w}_{\cdot, \text{m}} \in \mathbb{R}^{\bar{n}}$ and $\bm{w}_{\cdot, \text{v}} \in \mathbb{R}^{\bar{n}\bar{n}}$ denote weighting vectors of the variance accounted for (VAF) metric of mean and covariance, respectively. Hence, each element $m^{\text{VAF}}_i$ of $\bm{m}^{\text{VAF}} \in \mathbb{R}^{\bar{n}}$ is computed by
\begin{equation}
	m^{\text{VAF}}_i = \left(1 - \frac{\sum_{t=0}^{N} \left(\left(\E{\bm{M}\bm{x}_t}\right)_i - \hat{m}_{i,t}\right)^2}{\sum_{t=0}^{N} \left(\hat{m}_{i,t} - \frac{1}{N+1}\sum_{t}\hat{m}_{i,t}\right)^2}\right) \label{eq:VAFEW}
\end{equation}
and each element $\Omega_{ij}^{\text{VAF}}$ of $\bm{\Omega}^{\text{VAF}} \in \mathbb{R}^{\bar{n} \times \bar{n}}$ by
\begin{equation}
	\Omega^{\text{VAF}}_{ij} = \left(1 - \frac{\sum_{t=0}^{N} \left(\left(\bm{M}\bm{\Omega}^{\bm{x}}_{t}\bm{M}^{\mathrm{T}}\right)_{ij} - \hat{\Omega}_{ij,t}^{\bm{x}^{\ast}}\right)^2}{\sum_{t=0}^{N} \left(\hat{\Omega}_{ij,t}^{\bm{x}^{\ast}} - \frac{1}{N+1}\sum_{t}\hat{\Omega}_{ij,t}^{\bm{x}^{\ast}}\right)^2}\right) \label{eq:VAFCOV}.
\end{equation}
Due to the chosen VAF metric, $m^{\text{VAF}}_i \in (-\infty,1]$ and $\Omega^{\text{VAF}}_{ij} \in (-\infty,1]$ hold where a value of $1$ corresponds to a perfect fit between the mean or covariance resulting from the current guess of $\bm{s}$ and $\bm{\sigma}$ and the mean or covariance of the ground truth data. In $J_{\ISOC}$~\eqref{eq:J_GS}, the weighting vectors $\bm{w}_{\cdot, \text{m}} \in \mathbb{R}^{\bar{n}}$ and $\bm{w}_{\cdot, \text{v}} \in \mathbb{R}^{\bar{n}\bar{n}}$ enable a varying weighting between mean and covariance VAF values as well as between these values of different states. The denominator in \eqref{eq:J_GS} ensures that $J_{\ISOC} \in (-\infty,1]$ still holds. Hence, the direct optimization aims at maximizing $J_{\ISOC}$: $\max \limits_{\bm{s},\bm{\sigma}} \left( J_{\ISOC} \left( \hat{\bm{m}}_t, \hat{\bm{\Omega}}^{\bm{x}^{\ast}}_t, \E{\bm{x}_t}, \bm{\Omega}^{\bm{x}}_t \right) \right) = \max \limits_{\bm{s},\bm{\sigma}} \left( J_{\ISOC} \left( \hat{\bm{m}}_t, \hat{\bm{\Omega}}^{\bm{x}^{\ast}}_t, \bm{\sigma}, \bm{s} \right) \right)$. In order to evaluate $J_{\ISOC}$, $\E{\bm{x}_t}$ and $\bm{\Omega}_t$ need to be calculated for a current guess of $\bm{s}$ and $\bm{\sigma}$ which requires the computation of the solution to the forward optimal control problem. Hence, for the optimization problem a typical bi-level-based structure results which are common in deterministic IOC as well (see e.g. \cite{Mombaur.2010}). In order to facilitate the computability of the bi-level-based optimization problem, we introduce an alternating descent approach motivated by the natural division of the optimization variables into cost function $\bm{s}$ and noise parameters $\bm{\sigma}$\cite{Bezdek.2002}. Computing the best estimate~$\tilde{\bm{s}}^{(l)}$ for a given estimate~$\tilde{\bm{\sigma}}^{(l-1)}$ and the best estimate~$\tilde{\bm{\sigma}}^{(l)}$ for $\tilde{\bm{s}}^{(l)}$ in each iteration~$l$ yields the iterative scheme shown in Fig.~\ref{fig:general_algorithm}. In each of the two steps in one iteration~$l$ a bi-level optimization problem results. The upper level is solved via a derivative-free optimization method to circumvent computationally expensive and maybe unreliable approximations of derivatives. For our algorithm proposed in this paper, we introduce a specially designed grid search in the next subsection. Due to the existence of a huge amount of local minima in such direct IOC optimization problems, a global optimization method is chosen. For the lower level, the results of Section~\ref{sec:isoc_problems} are used. With a given $\bm{s}$ and $\bm{\sigma}$ that need to be evaluated in the upper level, $\bm{K}_t$ and $\bm{L}_t$ are computed via \eqref{eq:LQGEstimatorK} and \eqref{eq:LQGControlLawL} (LQG model) or \eqref{eq:SensoEstimatorK} and \eqref{eq:SensoControlLawL} (LQS model). Then, our newly introduced Lemma~\ref{lemma:LQGSolution} (LQG model) and Lemma~\ref{lemma:SensoSolution} (LQS model) lead to $\E{\bm{x}_t}$ and $\bm{\Omega}^{\bm{x}}_t$.

Remarkably, our algorithm omits a preparation step, where filter or control matrices are identified on the basis of the ground truth trajectories first, e.g. as necessary for the inverse algorithm for the LQG control model in \cite{Priess.2014}.
\comment{zukünftige Anpassungen am Algorithmus (kleinere Verbesserungen aktuell bereits umgesetzt, u.a. Löschen der Max-Bildung): Test Ersatz der Grid Search für upper level optimization durch andere (global) derivative-free optimization (z.B. classical pattern search, state-of-the-art/open-source approach oder auch matlab solver) - ggf. kann Vorteil der Parallelisierung sowie Aufteilung in Parametermengen beibehalten werden und gleichzeitig Konvergenz besser/schneller sein und auch Aussagen dazu getätigt werden (aktuell kann es zu einem Springen zwischen fast gleich guten grid points kommen - Konvergenz wird durch outer loop geforced oder durch unterschiedliche delta-Werte); Test andere (suboptimale) Lösungsverfahren im lower level für LQS model; Test derivative-based optimization methods mit dann approximierten derivatives (ggf. in globaler Variante)}
\comment{future work: vgl. comment before, Konvergenzbeweis two-step procedure (Ansatzpunkt: Konvergenz eines alternating descent und daraus ggf. folgende Bedingungen an Gütemaß), Konvergenzanalyse Gesamtalgorithmus, Anwendung Algorithmus auf Daten Dieter (und damit auch Tracking-Probleme - nächstes Paper), Journal(s) dann bei komplettierten Verbesserungen, inkl. Aussagen zur Gesamtkonvergenz und Ergebnissen zu klinischen menschlichen Bewegungen, PIIRL-based ISOC algorithm}
\tdd{
\begin{itemize}
	\item direct approach
	\item Strukturbild, das two-step procedure zeigt und in jedem step bi-level procedure ähnlich zu Mombaur - anmerken, dass später grid search genutzt wird in each step for upper level optimization (other derivative-free optimizations can be used)
	\item Erläuterung Anpassungen, um eine hinreichend große Anzahl an Parametern identifizierbar zu machen (neben two-step procedure, spezifisches Design grid-search based algorithm, parameter spaces, and lower level is realized by lemmata from previous section)
\end{itemize}
}

\subsection{Grid-Search-based Bi-Level Inverse Stochastic Optimal Control Algorithm} \label{subsec:gridsearch_isoc}
\begin{algorithm}[t] 
	\caption{Bi-Level-based Inverse Stochastic Optimal Control.}
	\label{algo:gridSearchAlgorithm}
	\DontPrintSemicolon
	\SetInd{0.25em}{0.4em}
	\KwIn{$\hat{\bm{m}}_t$, $\hat{\bm{\Omega}}_t^{\bm{x}^{\ast}}$, $[a_{\sigma_i},\, b_{\sigma_i}], \forall i \in \{1,\dots,\Sigma\}$, $[a_{s_i},\, b_{s_i}], \forall i \in \{1,\dots,S\}$, $\bar{\gamma}_{l}$, $l_{\max}$}
	\KwOut{$\tilde{\bm{\sigma}}$, $\tilde{\bm{s}}$}
	
	Set $l = 1$\;
	Set $\tilde{\sigma}_i^{(0)} = 0 ,  \forall i \in \{1,\dots,\Sigma\}$ (yields $\tilde{\bm{\sigma}}^{(0)}$)\;
	Set $\tilde{s}_i^{(0)} = \frac{a_{s_i} + b_{s_i}}{2}, \forall i \in \{1,\dots,S\}$ (yields $\tilde{\bm{s}}^{(0)}$)\;
	\Repeat{$l > l_{\max}$}{
		\Task{Determine $\tilde{\bm{s}}^{(l)}$ subject to $\tilde{\bm{\sigma}}^{(l-1)}$}
		{
			Apply Algorithm~\ref{algo:gridSearchAlgorithmSubAlg} with $\bm{\theta}^{(0)}=\tilde{\bm{s}}^{(l-1)}$ and $\bm{\lambda}=\tilde{\bm{\sigma}}^{(l-1)}$ to get $\tilde{\bm{s}}^{(l)}=\bm{\theta}^{\ast}$
		}
		\Task{Determine $\tilde{\bm{\sigma}}^{(l)}$ subject to $\tilde{\bm{s}}^{(l)}$}
		{
			Apply Algorithm~\ref{algo:gridSearchAlgorithmSubAlg} with $\bm{\theta}^{(0)}=\tilde{\bm{\sigma}}^{(l-1)}$ and $\bm{\lambda}=\tilde{\bm{s}}^{(l)}$ to get $\tilde{\bm{\sigma}}^{(l)}=\bm{\theta}^{\ast}$
		}
		$b_{s_i} = \frac{1}{\bar{\gamma}_l}\left(b_{s_i} + \bar{\gamma}_l a_{s_i} - a_{s_i}\right), \forall i \in \{1,\dots,S\}$\;
		$b_{\sigma_i} = \frac{1}{\bar{\gamma}_l}\left(b_{\sigma_i} +\bar{\gamma}_l  a_{\sigma_i} - a_{\sigma_i}\right), \forall i \in \{1,\dots,\Sigma\}$\;
		$l \leftarrow l+1$\;	
	}
	\Return{$\tilde{\bm{\sigma}} = \tilde{\bm{\sigma}}^{(l_{\max})}$, $\tilde{\bm{s}} = \tilde{\bm{s}}^{(l_{\max})}$}
\end{algorithm}
\begin{algorithm}[t!]
	\caption{Grid-Search-based Bi-Level Optimization.}
	\label{algo:gridSearchAlgorithmSubAlg}
	\DontPrintSemicolon
	\SetInd{0.25em}{0.4em}
	\KwIn{$\hat{\bm{m}}_t$, $\hat{\bm{\Omega}}_t^{\bm{x}^{\ast}}$, $\bm{\theta}^{(0)}$, $\bm{\lambda}$, $[a_{\theta_i},\, b_{\theta_i}], \forall i \in \{1,\dots,\Theta\}$, $\mathcal{N}_{\bm{\theta}}$, $\bm{w}_{\bm{\theta}, \text{m}}$, $\bm{w}_{\bm{\theta}, \text{v}}$, $\mathcal{P}_{\bm{\theta}}$, $\bar{\gamma}_{\theta}$, $\delta_{\gamma,\bm{\theta}}$, $\delta_{\bm{\theta}}$, $v_{\bm{\theta},\max}$}
	\KwOut{$\bm{\theta}^{\ast}$}
		Set $v = 1$, $\gamma = 2$\;
		Set $j_{\max} = 1$, $J_{\ISOC}^{(j_{\max}),(0)} = J_{\ISOC}^{(j_{\max}),(-1)} = -\infty$\;
		\Repeat{$v > v_{\bm{\theta},\max}$ $\lor$ $\Big(\lvert J_{\ISOC}^{(j_{\max}),(v-1)} - J_{\ISOC}^{(j_{\max}),(v-2)} \rvert < \delta_{\bm{\theta}} \land \lvert J_{\ISOC}^{(j_{\max}),(v-1)} - J_{\ISOC}^{(j_{\max}),(v-3)} \rvert < \delta_{\bm{\theta}}\Big)$}{
			\ForEach{$p \in \mathcal{P}_{\bm{\theta}}$}{
				\Task{Initialize search space}
				{
					$a_{\theta_i}^{(v)} = \max\big(0, \theta_i^{(v-1)} - \frac{b_{\theta_i} - a_{\theta_i}}{\gamma}\big), \forall \theta_i \in p$\;
					$b_{\theta_i}^{(v)} = \theta_i^{(v-1)} + \frac{b_{\theta_i} - a_{\theta_i}}{\gamma}, \forall \theta_i \in p$\;
					$\begin{aligned}
						\mathcal{M}_{\theta_i} &= \bigg\{ a_{\theta_i}^{(v)},\, a_{\theta_i}^{(v)} + \frac{1}{\mathcal{N}_{\bm{\theta}} - 1} \left(b_{\theta_i}^{(v)} - a_{\theta_i}^{(v)}\right), \dots,\\
						&\hphantom{=} a_{\theta_i}^{(v)} + \frac{\mathcal{N}_{\bm{\theta}} - 2}{\mathcal{N}_{\bm{\theta}} - 1} \left(b_{\theta_i}^{(v)} - a_{\theta_i}^{(v)}\right),\, b_{\theta_i}^{(v)} \bigg\}, \forall \theta_i \in p
					\end{aligned}$\;
					$\mathcal{M}_{\theta_i} = \{\theta_{i}^{(v-1)}\}, \forall \theta_i \notin p$\;
				}
				\Task{Initialize search grid}
				{
					Define $\mathcal{G}^{(v)}$ as the set of $\mathcal{N}_{\bm{\theta}}^{\lvert p \rvert}$ different vectors $\bar{\bm{\theta}}^{(j)}$, where $\bar{\theta}^{(j)}_i \in \mathcal{M}_{\theta_i}$ ($\forall i \in \{1,\dots,\Theta\}$) holds for their elements\;
				}
				\Task{Evaluate forward solution at grid points}
				{
					\For{$j = 1,\, 2,\, \dots,\, \mathcal{N}_{\bm{\theta}}^{\lvert p \rvert}$}{
						Determine $\bm{K}_t$ and $\bm{L}_t$ with $\bm{\lambda}$ and $\bar{\bm{\theta}}^{(j)} \in \mathcal{G}^{(v)}$ by \eqref{eq:LQGEstimatorK}, \eqref{eq:LQGControlLawL} (LQG) or \eqref{eq:SensoEstimatorK}, \eqref{eq:SensoControlLawL} (LQS)\;
						Determine $\E{\bm{x}_t}$ and $\bm{\Omega}_t^{\bm{x}}$ by Lemma~\ref{lemma:LQGSolution} (LQG) or Lemma~\ref{lemma:SensoSolution} (LQS)\;
						Determine $J_{\ISOC}^{(j),(v)}$ with $\bm{w}_{\bm{\theta},\text{m}}$ and $\bm{w}_{\bm{\theta},\text{v}}$ by \eqref{eq:J_GS} using $\E{\bm{x}_t}$ and $\bm{\Omega}_t^{\bm{x}}$\;
					}
				}
				\Task{Determine $j_{\max}$}
				{
					$j_{\max} = \arg \max \limits_j \left(J_{\ISOC}^{(j),(v)}\right)$\;
					$\bm{\theta}^{(v)} = \bar{\bm{\theta}}^{(j_{\max})}$\;
				}
			}
			\If{$\lvert J_{\ISOC}^{(j_{\max}),(v)} - J_{\ISOC}^{(j_{\max}),(v-1)} \rvert < \delta_{\gamma, \bm{\theta}}$}{
				$\gamma = \bar{\gamma}_{\bm{\theta}} \cdot \gamma$\;
			}			
			$v \leftarrow v + 1$\;
		}
		\Return{$\bm{\theta}^{\ast} = \bm{\theta}^{(v-1)}$}
\end{algorithm} 
Building upon the general procedure of our approach in Fig.~\ref{fig:general_algorithm}, Algorithm~\ref{algo:gridSearchAlgorithm} describes the detailed steps with the two bi-level optimizations in each iteration $l$ realized by Algorithm~\ref{algo:gridSearchAlgorithmSubAlg} which is a specially designed grid search to keep the total number of evaluated grid points computationally tractable. In Algorithm~\ref{algo:gridSearchAlgorithmSubAlg}, $\bm{\theta}$ describes a placeholder for the optimized parameter type and is replaced by $\bm{s}$ or $\bm{\sigma}$ in the respective step of one iteration of Algorithm~\ref{algo:gridSearchAlgorithm}. The optimization is done for the non-zero elements of $\bm{s}^{\ast}$ and $\bm{\sigma}^{\ast}$ (see Assumption~\ref{assumption:LQGBasisVectorsNonZeroElements}). Algorithm~\ref{algo:gridSearchAlgorithmSubAlg} iterates over single grid searches performed on subsets of all parameters $\bm{s}$ or $\bm{\sigma}$. We divide the parameters~$\bm{s}$ and $\bm{\sigma}$ in sets, where $\mathcal{P}_{\bm{s}}$ denotes the set of sets with elements of $\bm{s}$ and $\mathcal{P}_{\bm{\sigma}}$ the corresponding set for $\bm{\sigma}$. Each element of $\bm{s}$ and $\bm{\sigma}$ is supposed to be in at least one set of $\mathcal{P}_{\bm{s}}$ or $\mathcal{P}_{\bm{\sigma}}$. For each element $p \in \mathcal{P}_{\bm{s}}$ (or $\mathcal{P}_{\bm{\sigma}}$) a standard grid search is performed, where the values of the parameters not contained in the current $p$ are kept constant. The initial grid size is determined by the initial lower and upper limits $a_{s_i} \in \mathbb{R}_{\geq 0}$ (and $a_{\sigma_i} \in \mathbb{R}_{\geq 0}$) and $b_{s_i}\geq a_{s_i}$ (and $b_{\sigma_i} \geq a_{\sigma_i}$) between which the $\mathcal{N}_{\bm{s}}\in \mathbb{N}$ (and $\mathcal{N}_{\bm{\sigma}}\in \mathbb{N}$) grid points are spanned. The subsequent intervals for the grid points are determined with the current optimal solution as center. The interval length and thus the grid size is shrunk (tuning parameter $\bar{\gamma}_{\bm{s}}, \bar{\gamma}_{\bm{\sigma}} \in \mathbb{R}_{>0}$) when the iteration over all single grid searches yields no further improvement, i.e. the improvement of $J_{\ISOC}$ is smaller than a threshold $\delta_{\gamma,\bm{s}} \in \mathbb{R}_{>0}$ (or $\delta_{\gamma,\bm{\sigma}} \in \mathbb{R}_{>0}$). Furthermore, the grid size is shrunk after each outer iteration $l$ of Algorithm~\ref{algo:gridSearchAlgorithm} (tuning parameter $\bar{\gamma}_l \in \mathbb{R}_{>0}$) to ensure convergence. In addition to reducing the number evaluated grid points, the proposed parameter sets can further improve the overall performance if prior knowledge is available. For example, if a set of parameters in $\bm{s}$ (or $\bm{\sigma}$) has a strong mutually correlating (positive or negative) influence on $J_{\ISOC}$, they are optimized together in one grid search. The evaluation of the grid points (lower level optimization) of one single grid search is done in parallel on the number of cores available. Algorithm~\ref{algo:gridSearchAlgorithmSubAlg} terminates if the improvement of $J_{\ISOC}$ in three following iterations is smaller than a threshold $\delta_{\bm{s}} \in \mathbb{R}_{>0}$ (or $\delta_{\bm{\sigma}} \in \mathbb{R}_{>0}$) or if a maximum number \mbox{$v_{\bm{s},\max} \in \mathbb{N}$} (or $v_{\bm{\sigma},\max} \in \mathbb{N}$) of iterations is reached. Finally, Algorithm~\ref{algo:gridSearchAlgorithm} ends after $l_{\max}$ iterations and $\tilde{\bm{\sigma}}^{(0)} = \bm{0}$ leads to a fitted deterministic model in the first bi-level optimization of outer iteration $l=1$.

\comment{durch in Simulationsbeispiel oder im Tracking-Fall durchgeführte Systemerweiterung ändert sich hier nichts, da diese Fälle als partially measurable case aufgefasst werden (Abweichung in den erweiterten Zustandsgrößen sollte sowieso null sein, daher werden die erweiterten Zustände nicht als measurable klassifiziert), vgl. Section~\ref{sec:isoc_problems}}
\comment{im späteren Simulationsbeispiel wird nur noch auf die Varianzen und nicht die Kovarianzen gefittet, sprich formal ergibt sich $\bm{w}_{\cdot, \text{v}}^{\mathrm{T}} \diag\left( \bm{\Omega}^{\text{VAF}}\right)$ als zweiter Summand im Zähler von \eqref{eq:J_GS}}
\comment{anschließend an Kommentare in Section~\ref{sec:isoc_problems}: bisher wird keine vollständige Überprüfung der Matrizen auf die Erfüllung ihrer Bedingungen an jedem grid point durchgeführt, lediglich erfolgt implizit eine Gewährleistung von Werten ausschließlich $\geq 0$ in $\tilde{\bm{\sigma}}$ und $\tilde{\bm{s}}$, vgl. $\max$-Bildungen (was auch logische Randbedingung ist) - hier muss deshalb jedoch auch darauf geachten werden, dass initiales Intervall für jeden Parameter groß genug gewählt wird, sonst konvergiert Algorithmus schnell in Nebenminimum nahe null, da Nullsetzen der unteren Grenze automatisch zu einer Verkleinerung der grid size führt - daher wurde $\max$-Bildung in v2 rausgenommen - mit Werten $\geq 0$ in $\tilde{\bm{s}}$ sind Kostenfunktionsmatrizen aber per definitionem symmetrisch und positiv semi-definit, formal müsste noch positive Definitheit von $\bm{R}$ überprüft werden - Kovarianzmatrizen sind bereits per definitionem symmetrisch und positiv semi-definit, positive Definitheit von $\bm{\Omega}^{\bm{\omega}}$ müsste formal ebenfalls noch überprüft werden - perspektivisch kann hier noch eine genauere Überprüfung der Werte an einem Grid Point durch geführt werden als bisher in v2 mit Garantie $\geq 0$}


\section{SIMULATION RESULTS} \label{sec:sim_results}
We provide simulation results for the LQG and LQS case in Subsection~\ref{subsec:results_lqg} and \ref{subsec:results_sensorimotor}, respectively. The applied example system and ground truth parameters are given in Subsection~\ref{subsec:sim_example} and the performance of the ISOC algorithm is evaluated by the metrics defined in Subsection~\ref{subsec:metrics}.

\subsection{Simulation Example} \label{subsec:sim_example}
Our simulation example is given by a planar 2D point-to-point human hand movement motivated by \cite{Todorov.2002}. The system state is defined as $\bm{x}^{\mathrm{T}}=\mat{p_x & p_y & \dot{p}_x & \dot{p}_y & f_x & f_y & g_x & g_y}$, where $p_x$, $p_y$ and $\dot{p}_x$, $\dot{p}_y$ describe position and velocity of the hand, which is modeled as point mass ($m=\SI{1}{kg}$). Moreover, $f_x$ and $f_y$ denote the resultant forces on the hand in each dimension which are the outcomes of second-order linear filters with the neural activation $u_x$ and $u_y$ as inputs. The time constants of the filters are chosen as $\tau_1 = \tau_2 = \SI{40}{ms}$. Discretizing the dynamic and muscle filter equations with $\triangle t = \SI{10}{ms}$ leads to $p_{x,t+1} = p_{x,t} + \triangle t \dot{p}_{x,t}$, $\dot{p}_{x,t+1} = \dot{p}_{x,t} + \frac{\triangle t}{m} f_{x,t}$, $f_{x,t+1} = \left(1-\frac{\triangle t}{\tau_2}\right) f_{x,t} + \frac{\triangle t}{\tau_2} g_{x,t}$ and $g_{x,t+1} = \left(1-\frac{\triangle t}{\tau_1}\right) g_{x,t} + \frac{\triangle t}{\tau_1} u_{x,t}$ from which the system matrices $\bm{A}$, $\bm{B}$ follow (the discretized equations for the second dimension follow analogously). The output matrix is $\bm{H} = \mat{\bm{I}_{6 \times 6} & \bm{0}_{6 \times 2}}$ and for the initial value $\bm{x}_0 = \bm{0}$ holds. The cost function is given by
\begin{equation}
	J = \E{\left(\bm{x}_N - \bm{x}_{\text{ref}}\right)^{\mathrm{T}}\bm{Q}_N\left(\bm{x}_N - \bm{x}_{\text{ref}}\right) + \sum_{t=0}^{N-1}\bm{u}_t^{\mathrm{T}}\bm{R}\bm{u}_t} \label{eq:J_example},
\end{equation}
where $N=41$, $\bm{R} = \sum_{i=1}^{2} s_{R,i} \bm{\psi}_i\bm{\psi}_i^{\mathrm{T}}$ ($s_{R,1}=s_{R,2}=\frac{1}{42} 10^{-5}$), $\bm{Q}_N = \sum_{i=1}^{6} s_{N,i}\bm{\phi}_i\bm{\phi}_i^{\mathrm{T}}$ ($s_{N,1} = s_{N,2} = 1$, $s_{N,3} = s_{N,4} = 0.04$, $s_{N,5} = s_{N,6} =0.0004$) and $\bm{x}_{\text{ref}}=\SI{0.1}{m} \bm{\phi}_1 + \SI{0.1}{m} \bm{\phi}_2$. Here, $\bm{\psi}_i$ and $\bm{\phi}_i$ denote the standard unit vectors of $\mathbb{R}^{2}$ and $\mathbb{R}^{8}$, respectively. For the LQG case, the scaling matrices for the additive noise processes are chosen as $\bm{\Sigma}^{\bm{\xi}} = 1.5\diag\left(\bm{\phi}_7 + \bm{\phi}_8\right)$ and $\bm{\Sigma}^{\bm{\omega}} = \diag\left(\mat{0.02 & 0.02 & 0.2 & 0.2 & 1 & 1}^{\mathrm{T}}\right)$. For the LQS case, the same scaling matrix $\bm{\Sigma}^{\bm{\omega}}$ is selected but $\bm{\Sigma}^{\bm{\xi}}=\bm{0}$ holds. Furthermore, we define the scaling matrices for the signal-dependent noise processes as $\bm{C}_1 = \sigma^{\bm{u}} \bm{B}$, $\bm{C}_2 = \sigma^{\bm{u}} \bm{B} \mat{-\bm{\psi}_2 & \bm{\psi}_1}$ ($\sigma^{\bm{u}} = 0.5$) and $\bm{D} = 0.1 \bm{H}$. Eq.~\eqref{eq:J_example} can be transformed into a form as \eqref{eq:LQGCost} by augmenting the system state with $p_{x,\text{ref}}$ and $p_{y,\text{ref}}$ as additional states with constant dynamics\cite{Todorov.2002}. We assume that only $p_x$, $p_y$, $\dot{p}_x$ and $\dot{p}_y$ are measured in the ground truth data ($\bm{M}=\mat{\bm{I}_{4 \times 4} & \bm{0}_{4 \times 4}}$), which could be realized in practice by motion capturing. 

Finally, we define the vectors $\bm{s}$ and $\bm{\sigma}$ as $\bm{s}^{\mathrm{T}} = \mat{s_{N,1} & \dots & s_{N,6} & s_{R,1} & s_{R,2}}$ and $\bm{\sigma}^{\mathrm{T}} = \mat{\sigma_{1}^{\bm{\xi}} & \dots & \sigma_{8}^{\bm{\xi}} & \sigma_{1}^{\bm{\omega}} & \dots & \sigma_{6}^{\bm{\omega}}}$ (LQG case) or $\bm{\sigma}^{\mathrm{T}} = \mat{\sigma_{1}^{\bm{\xi}} & \dots & \sigma_{8}^{\bm{\xi}} & \sigma_{1}^{\bm{\omega}} & \dots & \sigma_{6}^{\bm{\omega}} & \sigma^{\bm{u}} & \sigma^{\bm{x}}}$ (LQS case), where $\sigma_i^{\bm{\xi}}$ and $\sigma_i^{\bm{\omega}}$ denote the diagonal elements of $\bm{\Sigma}^{\bm{\xi}}$ and $\bm{\Sigma}^{\bm{\omega}}$, respectively. Here, compared to Section~\ref{sec:isoc_problems}, we consider that $\bm{\Sigma}^{\bm{\xi}}$ and $\bm{\Sigma}^{\bm{\omega}}$ are diagonal. We use $\bm{s}^{\ast}$ and $\bm{\sigma}^{\ast}$ defined from the numerical values before to simulate the ground truth data for our ISOC algorithm, i.e. $\hat{\bm{m}}_t = \E{\bm{M}\bm{x}^{\ast}_t}$ and $\hat{\bm{\Omega}}^{\bm{x}^{\ast}}_t = \bm{M}\bm{\Omega}^{\bm{x}^{\ast}}_t\bm{M}^{\text{T}}$ hold where $\E{\bm{x}^{\ast}_t}$ and $\bm{\Omega}^{\bm{x}^{\ast}}_t$ follow from Lemma~\ref{lemma:LQGSolution} or Lemma~\ref{lemma:SensoSolution} using $\bm{s}^{\ast}$ and $\bm{\sigma}^{\ast}$.
\comment{Hintergedanke: $\bm{\Sigma}^{\bm{\xi}}$ beschreibt additives Rauschen in Stellgröße im LQG-Fall als Alternative zum multiplikativen Rauschen im sensomotorischen Modell (deshalb Skalierungsmatrix $\bm{\Sigma}^{\bm{\xi}}$ in letzterem zu null gewählt, Rauschen in Stellgröße wird hier eben durch multiplikativen Teil dargestellt); Varianz und damit Unsicherheit der Perzeption nimmt zwischen Position und Geschwindigkeit um eine Größenordnung zu und dann nochmal zur Kraft, da Position mit größter Sicherheit vom Menschen erfasst wird, Geschwindigkeit deutlich schlechter und Kraft kann nur propriozeptisch wahrgenommen werden, daher nochmals größere Unsicherheit (bei sensomotorischem Modell kommt dann multiplikatives Rauschen zusätzlich dazu) - Gütemaß beschreibt Abweichung von Sollposition (kartesische Größen in Zielkoordinaten, Kraft und Geschwindigkeit null), bei gleichzeitige Bestrafung der neuronalen Stellgröße (effort) über die Bewegung, da aber schnelle Bewegungen betrachtet werden und Zielposition relativ gut erreicht werden soll, hier geringes Gewicht}

\subsection{Evaluation Metrics} \label{subsec:metrics}
The VAF metrics $\bm{m}^{\text{VAF}}$~\eqref{eq:VAFEW} and $\bm{\Omega}^{\text{VAF}}$~\eqref{eq:VAFCOV}, where $\tilde{\bm{x}}_t$ results from $\tilde{\bm{s}}$ and $\tilde{\bm{\sigma}}$ via $\tilde{\bm{K}}_t$ and $\tilde{\bm{L}}_t$, are used to asses if Problem~\ref{problem:ISOCLQG} or \ref{problem:ISOCSenso} is solved. Furthermore, we introduce the parameter errors
\begin{align}
	\triangle_i^{\bm{s}} &= \left| 1 - \frac{\tilde{s}_i}{s_i^{\ast}} \frac{s_{N,1}^{\ast}}{\tilde{s}_{N,1}} \right|, \forall i \in \{1,\dots,S\} \label{eq:paramErrorS}, \\
	\triangle_i^{\bm{\sigma}} &= \left| 1 - \frac{\tilde{\sigma}_i}{\sigma_i^{\ast}} \right|, \forall i \in \{1,\dots,\Sigma\} \label{eq:paramErrorSigma}
\end{align}
to asses the difference between the estimated and ground truth values when the ground truth ones are non-zero. The normalization in \eqref{eq:paramErrorS} is due to the scaling ambiguity of the cost function parameters, well-known in deterministic IOC \cite{Molloy.2020}. This kind of ambiguity is absent for noise parameters.

\subsection{Results of the Linear-Quadratic Gaussian Case} \label{subsec:results_lqg}
The tuning parameters of Algorithm~\ref{algo:gridSearchAlgorithm} for its application to the LQG example system are chosen as: $\mathcal{N}_{\bm{\sigma}} = \mathcal{N}_{\bm{s}} = 8$, $\bar{\gamma}_l = \bar{\gamma}_{\bm{\sigma}} = \bar{\gamma}_{\bm{s}} = 2$, $v_{\bm{\sigma},\max} = v_{\bm{s},\max} = 20$, $l_{\max} = 3$, $\delta_{\gamma,\bm{\sigma}} = \delta_{\gamma,\bm{s}} = 0.01$, $\delta_{\bm{\sigma}} = \delta_{\bm{s}} = 0.001$, $a_{\sigma_i} = 0$, $b_{\sigma_i} = 4$ ($\forall i \in \{1,\dots,\Sigma\}$), $a_{s_i} = 0$ ($\forall i \in \{1,\dots,S\}$), $b_{s_i} = 4$ ($i \in \{1,2\}$), $b_{s_i} = 0.4$ ($i \in \{3,4\}$), $b_{s_i} = 0.004$ ($i \in \{5,6\}$), $b_{s_i} = 4\cdot 10^{-6}$ ($i \in \{7,8\}$), $\bm{w}_{\bm{\sigma},\text{m}}^{\mathrm{T}} = \bm{w}_{\bm{s},\text{v}}^{\mathrm{T}} = \mat{0.1 & 0.1 & 0.1 & 0.1}$, $\bm{w}_{\bm{\sigma},\text{v}}^{\mathrm{T}} = \bm{w}_{\bm{s},\text{m}}^{\mathrm{T}} = \mat{0.9 & 0.9 & 0.9 & 0.9}$. Here, the second summand in the numerator of \eqref{eq:J_GS} is $\bm{w}_{\cdot, \text{v}}^{\mathrm{T}} \diag\left( \bm{\Omega}^{\text{VAF}}\right)$. Finally, we select the parameter sets as $\mathcal{P}_{\bm{\sigma}} = \{ \{\sigma_1, \sigma_3, \sigma_5, \sigma_7\}, \{\sigma_2, \sigma_4, \sigma_6, \sigma_8\}, \{\sigma_9, \sigma_{11}, \sigma_{13}\}, \{\sigma_{10}, \sigma_{12}, \\ \sigma_{14}\} \}$ and $\mathcal{P}_{\bm{s}} = \{ \{s_1, s_3, s_5, s_7\}, \{s_2, s_4, s_6, s_8\} \}$. The algorithm and simulations are implemented in Matlab.
Fig.~\ref{fig:LQG_y} shows that our algorithm solves Problem~\ref{problem:ISOCLQG} since the VAF values for mean and variance of the measurable states are $\geq 99.8\,\%$. The computation time on a Ryzen 7 5800X with 8 parallel cores was $\SI{247}{s}$.
\comment{Rechenzeit würde bei geschätzten Momenten aus Einzelsimulationen (5000 nötig für Übereinstimmung mit Momenten der Lemmata) um ca. $\SI{273.6}{h}$ ansteigen; Rechenzeit $T_C$ kann folgendermaßen abgeschätzt werden: $T_C = N_G (T_{FS}+T_{EV})$ mit der Anzahl $N_G$ der insgesamt über alle Iterationen evaluierten grid points, der Rechenzeit für die Berechnung der Vorwärtslösung $T_{FS}$ sowie der Rechenzeit für die Bestimmung der Momente, welche im upper level ausgewertet werden, $T_{EV}$; mit $T_{EV,sim} = c T_{EV,lemmata}$ folgt für die Rechenzeitzunahme $N_G (c-1) T_{EV,lemmata}$; $T_{EV,lemmata}$ und $T_{EV,sim}$ sind dabei unabhängig von den konkreten Werten an den grid points (für LQG und LQS), lediglich $T_{FS}$ variiert im LQS-Fall relativ stark an verschiedenen grid points und ist prinzipiell höher als im LQG-Fall (auch $N_G$ ist im LQS-Fall höher, da für gute Ergebnisse eine größere Genauigkeit in den grid points und mehr Iterationen benötigt werden); damit kann aber $T_{EV,lemmata}$ und $T_{EV,sim}$ aus einer Analyse an den optimalen Werten bestimmt werden und mit $N_G$ ergibt sich Rechenzeitzunahme}
\begin{figure}[t]
	\centering
	\includegraphics[width=3in]{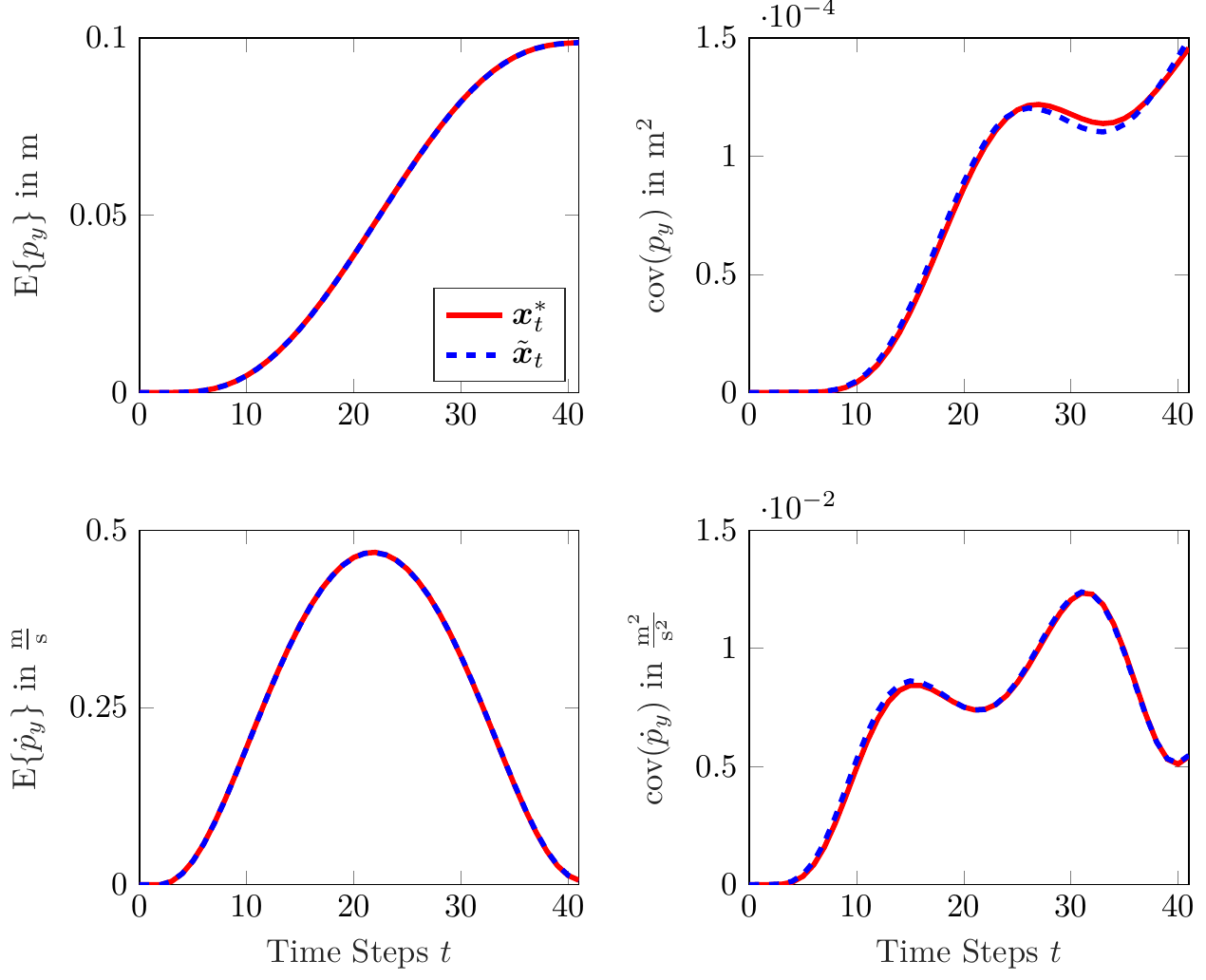}
	\caption{Mean and variance of $p_y$ and $\dot{p}_y$ (LQG). The values of $\bm{x}_t^{\ast}$ (resulting from $\bm{s}^{\ast}$ and $\bm{\sigma}^{\ast}$) and $\tilde{\bm{x}}_t$ (resulting from $\tilde{\bm{s}}$ and $\tilde{\bm{\sigma}}$) are compared. The VAF metrics are $1$ and $0.999$ in the first and second row. The corresponding values for $p_x$ are $1$ and $0.999$ and for $\dot{p}_x$ $1$ and $0.998$.}
	\label{fig:LQG_y}
\end{figure}
\begin{figure}[t]
	\centering
	\includegraphics[width=3in]{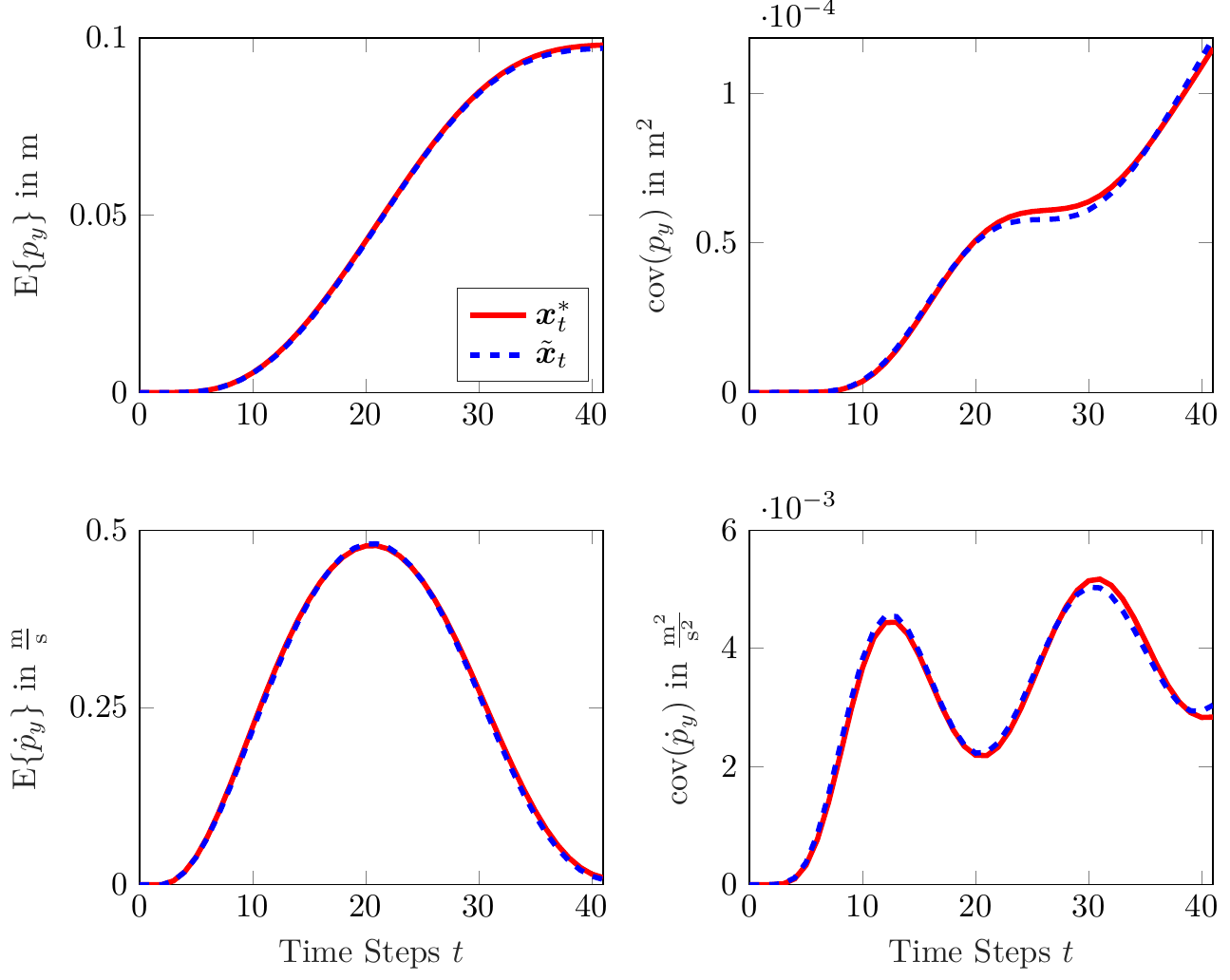}
	\caption{Mean and variance of $p_y$ and $\dot{p}_y$ (LQS). The VAF metrics are $1$ and $0.998$ in the first and $1$ and $0.996$ in the second row. The corresponding values for $p_x$ are $1$ and $0.999$ and for $\dot{p}_x$ $0.999$ and $0.965$.}
	\label{fig:Senso_y}
\end{figure}

In Table~\ref{tab:parameter_results} the parameters $\tilde{\bm{\sigma}}$ and $\tilde{\bm{s}}$ are compared with their ground truth values $\bm{\sigma}^{\ast}$ and $\bm{s}^{\ast}$. The high accuracy of estimating the parameters of the additive noise process $\bm{\xi}_t$ is remarkable which is due to their strong influence on the variance of the measurable states. In addition, the cost function parameters $s_{N,1}$, $s_{N,3}$ and $s_{R,1}$ are identified correctly up to the known scaling ambiguity. Regarding the parameters $s_{N,2}$, $s_{N,4}$ and $s_{R,2}$, an additional ambiguity can be seen. In the LQG case, the 2D system is composed by two 1D systems. Hence, to describe the moments of the measurable states it is sufficient that the ratios between $s_{N,1}$, $s_{N,3}$, $s_{R,1}$ and $s_{N,2}$, $s_{N,4}$, $s_{R,2}$ are estimated correctly but not with the same reference value necessarily. Thus, changing the normalizing parameter in \eqref{eq:paramErrorS} to $s_{N,2}$ leads to $\triangle_2^{\bm{s}}=0$, $\triangle_4^{\bm{s}}=0.05$ and $\triangle_8^{\bm{s}}=0.01$. However, the parameters of the additive observation noise are estimated with high errors. This can be traced back to the filter equations. Since they compensate noisy observations, the underlying noise parameters cannot be determined from realizations of $\bm{x}_t^{\ast}$. 
\comment{bzgl. noise parameters des Ausgangsrauschens: hier kann man dann auch weiter schlussfolgern, dass es damit im Allgemeinen schwer/unmöglich ist, die Parameter des Rauschens der human perception zu identifizieren, wenn nur von Messungen der resultierenden Zustände ausgegangen wird - logischerweise wird das Rauschen in der human perception gerade durch die filter equations möglichst kompensiert (sprich Mensch kompensiert dieses durch Modell der Umwelt) - daher entsteht bzgl. dieser Parameter große Ambiguität (größeren Einfluss könnte dann gerade wieder das interne Rauschen haben) - konkret entsteht im ISOC-Fall eine zweistufige Ambiguität: mehrere Parameter können zu identischen control/filter matrices führen (analog zu deterministischem IOC), unterschiedliche control/filter matrices können mit ihren jeweiligen Parametern aber auch zu den gleichen Momenten führen (neu bei ISOC) - das ist gerade hier im LQG-Fall bzgl. der filter matrices der Fall (verschiedene filter matrices zu verschiedenen observation noise parameters = identische Momente) - formal kann das dann gerade an den obigen Lemmata analysiert werden - kann dann weiter geschlussfolgert werden, dass im fully observable case ebenfalls Parameter existieren, die gleiche Momente erzeugen (auch an Formeln zu analysieren/belegen), kann für ISOC Vereinfachung auf fully observable case gemacht werden (PIIRL-Anwendung für Performance-Verbesserung und theoretische Analyse) - observation noise parameter und damit human perception erstmal nicht ohne weiteres identifizierbar und auch kein relevanter Einfluss auf Resultat des Modells, welches für design supporting automation interessant ist}

\subsection{Results of the Linear-Quadratic Sensorimotor Case} \label{subsec:results_sensorimotor}
The tuning parameters to apply Algorithm~\ref{algo:gridSearchAlgorithm} to the LQS example are the same as in the subsection before. Only $\mathcal{N}_{\bm{\sigma}} = \mathcal{N}_{\bm{s}}=10$ holds and the parameter sets of the noise parameters are adapted: $\mathcal{P}_{\bm{\sigma}} = \{ \{\sigma_1, \sigma_3\}, \{\sigma_2, \sigma_4\}, \{\sigma_5, \sigma_7, \sigma_{15}\}, \{\sigma_6, \sigma_8, \sigma_{15}\}, \{\sigma_9, \sigma_{11},\\\sigma_{13}, \sigma_{16}\}, \{\sigma_{10}, \sigma_{12}, \sigma_{14}, \sigma_{16}\} \}$.
Fig.~\ref{fig:Senso_y} shows that our algorithm finds parameters that solve Problem~\ref{problem:ISOCSenso} as well. All VAF values are $\geq 99.6\,\%$ except the one for the variance of $\dot{p}_x$ which still has a VAF $\geq 96.5\,\%$. The computation time was $\SI{19.9}{h}$. The iterative calculation of $\bm{K}_t$ and $\bm{L}_t$ at each grid point is one reason for the computational complexity. In addition, to achieve an accuracy comparable to the LQG case, the total number of grid points evaluated is greater.
Table~\ref{tab:parameter_results} depicts the parameter errors for the LQS example as well. As in the LQG case, the exact values of the observation noise parameters have low influence on the measurable states and thus, show high error measures\comment{auch hier stimmen wie im LQG-Fall die control und filter matrices der reference nicht mit den der identifizierten Parameter überein (bei LQG stimmen jedoch die control matrices überein - separation theorem), d.h. auch hier tritt die oben erläuterte zweistufige Ambiguität auf, die sich hier jedoch auch auf die control matrices ausweitet - logisch, da hier Kombination aus control/filter matrices aufgrund ihrer gegenseitigen Abhängigkeit das observation noise kompensieren (andere observation noise parameter, andere filter matrices, andere control matrices) - die anderen control matrices können hier im LQS-Fall dann gerade dazu führen, dass die Gütemaßparameter hier größere Fehler aufweisen trotz sehr hoher VAF-Maße, vergleichbar zu LQG model}. Remarkably, the scaling parameter $\sigma^{\bm{u}}$ of the control-dependent noise is estimated with high accuracy (relative error of $3\,\%$). Consequently, the control-dependent noise is the most important part of the LQS model to achieve its characteristic mean and variance curves of the measurable states. Specifically, compared to the LQG model, the peak of $\covs{p_y}$ at $t=26$ in Fig.~\ref{fig:LQG_y} disappears in Fig.~\ref{fig:Senso_y} and the values of $\covs{p_y}$ are lower over the complete time horizon in case of the LQS model. Furthermore, the peaks of $\E{\dot{p}_y}$ and $\covs{\dot{p}_y}$ appear around two steps earlier and the peaks of $\covs{\dot{p}_y}$ have nearly the same height in the LQS case.
\comment{Additive noise in force and neural activation ($\sigma^{\bm{\xi}}_{5-8}$) seems to result in very similar moments of the measurable states since their corresponding values are non-zero and higher for the $p_x$-components ($\sigma^{\bm{\xi}}_{5}$, $\sigma^{\bm{\xi}}_{7}$), i.e. the quantity with slightly worse variance matching of its velocity. Thus, the optimization algorithm has difficulties to resolve this ambiguity and can converge before. Regarding the cost function parameters, $s_{N,1}$ and $s_{N,2}$ are the most influential once, followed by $s_{R,1}$ and $s_{R,2}$ where however only the magnitude needs to be identified correctly.}
\comment{gemäß der im LQG-Fall eingeführten Rechenzeitanalyse würde im LQS-Fall die Rechenzeit bei berechneten Momenten aus Einzelsimulationen um $\SI{931.9}{h}$ steigen}
\tdd{
	\begin{itemize}
		\item nochmalige Ausführung der Optimierung mit gleichen Parametern, Bestimmung $N_G$ und damit Berechnung Rechenzeitzunahme
	\end{itemize}
}

\begin{table}[t]
	\centering
	\caption{Parameter errors according to \eqref{eq:paramErrorS} and \eqref{eq:paramErrorSigma} achieved by Algorithm~\ref{algo:gridSearchAlgorithm}. Values marked with $^{\ast}$ denote raw values calculated by our algorithm since the corresponding ground truth value is zero.}
	\label{tab:parameter_results}
	\begin{tabularx}{0.485\textwidth}{|>{\centering\arraybackslash}X|>{\centering\arraybackslash}X|>{\centering\arraybackslash}X||>{\centering\arraybackslash}X|>{\centering\arraybackslash}X|>{\centering\arraybackslash}X||>{\centering\arraybackslash}X|>{\centering\arraybackslash}X|>{\centering\arraybackslash}X|}
		\hline	
		& LQG & LQS & & LQG & LQS & & LQG & LQS \\
		\hline	
		$\sigma_1^{\bm{\xi}}$ &	$0^{\ast}$ & $0^{\ast}$ & $s_{N,1}$ & $0$ & $0$ & $\sigma_1^{\bm{\omega}}$ & $0.21$ & $684$ \\
		\hline	
		$\sigma_2^{\bm{\xi}}$ & $0^{\ast}$ & $0^{\ast}$ & $s_{N,2}$ & $0.66$ & $0$ & $\sigma_2^{\bm{\omega}}$ & $22.2$ & $510$ \\
		\hline	
		$\sigma_3^{\bm{\xi}}$ & $0^{\ast}$ & $0^{\ast}$ & $s_{N,3}$ & $0.07$ & $3.94$ & $\sigma_3^{\bm{\omega}}$ & $0.57$ &	$24.7$ \\
		\hline	
		$\sigma_4^{\bm{\xi}}$ &	$0^{\ast}$ & $0^{\ast}$ & $s_{N,4}$ & $0.58$ & $0.98$ & $\sigma_4^{\bm{\omega}}$ & $0.03$ &	$0.14$ \\
		\hline	
		$\sigma_5^{\bm{\xi}}$ &	$0.18^{\ast}$ & $0.30^{\ast}$ & $s_{N,5}$ & $0.05$ & $0.64$ & $\sigma_5^{\bm{\omega}}$ & $0.15$ & $0.56$ \\
		\hline	
		$\sigma_6^{\bm{\xi}}$ & $0.19^{\ast}$  & $0.09^{\ast}$ & $s_{N,6}$ & $0.05$ & $0.64$ & $\sigma_6^{\bm{\omega}}$ & $0.13$ & $0.25$ \\
		\hline	
		$\sigma_7^{\bm{\xi}}$ &	$0.03$ & $0.68^{\ast}$ & $s_{R,1}$ & $0.04$ & $0.41$ & $\sigma^{\bm{x}}$ & - &	$1.78$ \\
		\hline	
		$\sigma_8^{\bm{\xi}}$ &	$0.03$ & $0.32^{\ast}$ & $s_{R,2}$ & $0.65$ & $0.41$ & & &	\\
		\hline	
		$\sigma^{\bm{u}}$ & - & $0.03$ & & & & & &	\\
		\hline	
	\end{tabularx}
\end{table}


\section{CONCLUSION} \label{sec:conclusion}
In this paper, we propose formal definitions of the inverse problem of the LQG and LQS control model for the first time. Solving these ISOC problems is highly relevant for the identification of the unknown parameters of these SOC models, namely weighting matrices of the cost function and covariance matrices of the noise processes, to analyze the optimality principles underlying human movements. Furthermore, we introduce a new bi-level-based algorithm, which iteratively estimates cost function and noise parameters, to solve both ISOC problems. Simulation examples show that mean and variance of system states obtained from parameters determined by our ISOC algorithm predominantly yield VAF values $\geq 99\,\%$ compared to ground truth data. Since the simulation results are very promising, we look at the application of our algorithm to real human measurement data in the next step. In addition, we investigate extensions of our approach to nonlinear systems to overcome the limitations of linear models in describing real human movements.
\tdd{
	\begin{itemize}
		\item Zusammenfassung, inkl. verkaufen, sprich der entwickelte Ansatz zur Identifikation der unbekannten Parameter (Kostenfunktions- und Rauschparameter, beide wichtig sogar für mittleres Verhalten im state-of-the-art sensomotorischen Modell nach \cite{Todorov.2002}; letztere stets für Beschreibung Variabilität) stochastischer Modelle menschlicher Bewegungen bildet die Grundlage für die Verifikation dieser Modelle im Rahmen von Studien, damit auch dem Verständnis des Menschen und seiner Bewegungen (und deren grundlegenden Prinzipien), was schließlich auch der zwingende Ausgangspunkt zum Design einer supporting automation ist (dabei können auch direkt die identifizierten stochastischen Optimalregelungsmodelle für das control design genutzt werden)
		\item Ausblick, v.a. dass in further work gerade drei Aspekte zentrale Rolle spielen: Konvergenzbeweis two-step procedure, Ersatz bi-level optimization durch indirect approaches (e.g. IRL methods), Anwendung auf reale Daten menschlicher Bewegungen - hier vorgestellter Ansatz bildet somit die Grundlage für eine Vielzahl weiterer Arbeiten in unterschiedlichste Richtungen (weiteres Verkaufsargument)
	\end{itemize}
}





\section*{APPENDIX: PROOF OF LEMMA~\ref{lemma:SensoSolution}}
\begin{proof}
	With $\bm{u}_t = -\bm{L}_t \hat{\bm{x}}_t$ in \eqref{eq:SensoDynamics} and in the filter equation for $\hat{\bm{x}}_{t+1}$, \eqref{eq:SensoSolutionEW} results by taking advantage of the independence between $\bm{\varepsilon}_t$ and $\hat{\bm{x}}_t$ as well as between $\bm{\epsilon}_t$ and $\bm{x}_t$. Using the same expressions for $\bm{x}_{t+1}$ and $\hat{\bm{x}}_{t+1}$ to set up the covariance, we derive
	\begin{align}
		&\mat{\bm{\Omega}_{t+1}^{\bm{x}} \!\!& \bm{\Omega}_{t+1}^{\bm{x}\hat{\bm{x}}} \\ \bm{\Omega}_{t+1}^{\hat{\bm{x}}\bm{x}} \!\!& \bm{\Omega}_{t+1}^{\hat{\bm{x}}}} = \covs{\mat{\bm{x}_{t+1} \\ \hat{\bm{x}}_{t+1}}} \nonumber \\
		&= \bm{\mathcal{A}}_t \mat{\bm{\Omega}_{t}^{\bm{x}} \!\!\!\!& \bm{\Omega}_{t}^{\bm{x}\hat{\bm{x}}} \\ \bm{\Omega}_{t}^{\hat{\bm{x}}\bm{x}} \!\!\!\!& \bm{\Omega}_{t}^{\hat{\bm{x}}}} \bm{\mathcal{A}}_t^{\mathrm{T}} \nonumber \\
		&\hphantom{=} + \E{\mat{-\sum_{i} \varepsilon_t^{(i)} \bm{C}_i \bm{L}_t \hat{\bm{x}}_t \\ \bm{K}_t \sum_{i} \epsilon_t^{(i)} \bm{D}_i \bm{x}_t} \mat{-\sum_{i} \varepsilon_t^{(i)} \bm{C}_i \bm{L}_t \hat{\bm{x}}_t \\ \bm{K}_t \sum_{i} \epsilon_t^{(i)} \bm{D}_i \bm{x}_t}^{\mathrm{T}}} \nonumber \\
		&\hphantom{=} + \E{\mat{\bm{\xi}_{t} \\ \bm{\eta}_{t} + \bm{K}_t \bm{\omega}_t} \mat{\bm{\xi}_{t} \\ \bm{\eta}_{t} + \bm{K}_t \bm{\omega}_t}^{\mathrm{T}}} \label{eq:proof2_1}
	\end{align}
	by exploiting the following independences: $\bm{\varepsilon}_t$ to $\hat{\bm{x}}_t$; $\bm{\epsilon}_t$ to $\bm{x}_t$; $\bm{\xi}_t$ to $\bm{x}_t$, $\hat{\bm{x}}_t$, $\bm{\varepsilon}_t$ and $\bm{\epsilon}_t$; $\bm{\eta}_t$ to $\bm{x}_t$, $\hat{\bm{x}}_t$, $\bm{\varepsilon}_t$ and $\bm{\epsilon}_t$; $\bm{\omega}_t$ to $\bm{x}_t$, $\hat{\bm{x}}_t$, $\bm{\varepsilon}_t$ and $\bm{\epsilon}_t$.
	From the independence of $\bm{\xi}_t$ to $\bm{\eta}_{t}$ and $\bm{\omega}_t$
	\begin{equation}
		\mat{\bm{\Omega}^{\bm{\xi}} \!\!& \bm{0} \\ \bm{0} \!\!& \bm{\Omega}^{\bm{\eta}} + \bm{K}_t \bm{\Omega}^{\bm{\omega}} \bm{K}_t^{\mathrm{T}}} \label{eq:proof2_2}
	\end{equation} 
	follows for the third summand in \eqref{eq:proof2_1}. Each element of the matrix of the second summand in \eqref{eq:proof2_1} can be simplified by considering the independence of $\bm{\varepsilon}_t$ and $\bm{\epsilon}_t$ to $\bm{x}_t$ and $\hat{\bm{x}}_t$ as well as $\E{\varepsilon_t^{(i)} \varepsilon_t^{(j)}} = \E{\epsilon_t^{(i)} \epsilon_t^{(j)}} = \delta_{ij}$ ($\delta_{ij} = 1$ for $i=j$, $\delta_{ij} = 0$ for $i\neq j$) and $\E{\varepsilon_t^{(i)} \epsilon_t^{(j)}} = 0$ ($\forall i,j$):
	\begin{align}
		&\E{\sum_{i} \varepsilon_t^{(i)} \bm{C}_i \bm{L}_t \hat{\bm{x}}_t \hat{\bm{x}}_t^{\mathrm{T}} \bm{L}_t^{\mathrm{T}} \sum_{j} \bm{C}_j^{\mathrm{T}} \varepsilon_t^{(j)}} \nonumber \\
		&= \sum_{i} \bm{C}_i \bm{L}_t \E{\hat{\bm{x}}_t \hat{\bm{x}}_t^{\mathrm{T}}} \bm{L}_t^{\mathrm{T}} \bm{C}_i^{\mathrm{T}} \label{eq:proof2_3} \\
		&\E{-\sum_{i} \varepsilon_t^{(i)} \bm{C}_i \bm{L}_t \hat{\bm{x}}_t \bm{x}_t^{\mathrm{T}} \sum_{j} \bm{D}_j^{\mathrm{T}} \epsilon_t^{(j)} \bm{K}_t^{\mathrm{T}}} = \bm{0} \\
		&\E{-\bm{K}_t \sum_{i} \epsilon_t^{(i)} \bm{D}_i \bm{x}_t \hat{\bm{x}}_t^{\mathrm{T}} \bm{L}_t \sum_{j} \bm{C}_j^{\mathrm{T}} \varepsilon_t^{(j)}} = \bm{0} \\
		&\E{\bm{K}_t \sum_{i} \epsilon_t^{(i)} \bm{D}_i \bm{x}_t \bm{x}_t^{\mathrm{T}} \sum_{j} \bm{D}_j^{\mathrm{T}} \epsilon_t^{(j)} \bm{K}_t^{\mathrm{T}}} \nonumber \\ 
		&= \sum_i \bm{K}_t \bm{D}_i \E{ \bm{x}_t \bm{x}_t^{\mathrm{T}} } \bm{D}_i^{\mathrm{T}} \bm{K}_t^{\mathrm{T}} \label{eq:proof2_4}.
	\end{align}
	With $\E{\hat{\bm{x}}_t \hat{\bm{x}}_t^{\mathrm{T}}} = \bm{\Omega}_{t}^{\hat{\bm{x}}} + \E{\hat{\bm{x}}_t}\E{\hat{\bm{x}}_t}^{\mathrm{T}}$ in \eqref{eq:proof2_3} and $\E{ \bm{x}_t \bm{x}_t^{\mathrm{T}} } = \bm{\Omega}_{t}^{\bm{x}} + \E{\bm{x}_t}\E{\bm{x}_t}^{\mathrm{T}}$ in \eqref{eq:proof2_4}, \eqref{eq:proof2_1} reduces to \eqref{eq:SensoSolutionCOV}. 
	The initial values follow analogously to the proof of Lemma~\ref{lemma:LQGSolution}.
\end{proof}
%
%


\bibliographystyle{IEEEtran}

\bibliography{IEEEabrv,References}

\end{document}